\documentclass[12pt,reqno]{amsart}
\usepackage{amssymb,latexsym}
\usepackage{amsmath}
\usepackage{amsthm}
\usepackage{graphicx}
\usepackage{hyperref}
\usepackage{titletoc}
\usepackage{pdfsync}
\usepackage{color,xcolor}
\usepackage{amsthm,amsmath,amssymb}

\usepackage{mathrsfs}

\numberwithin{equation}{section}
\newtheorem{theorem}{Theorem}[section]
\newtheorem{proposition}[theorem]{Proposition}
\newtheorem{lemma}[theorem]{Lemma}

\newtheorem*{thmA}{Theorem A}

\usepackage{multirow}
\usepackage[font=bf,aboveskip=15pt]{caption}
\usepackage[toc,page]{appendix}

\theoremstyle{definition}

\newtheorem{remark}[theorem]{Remark}

\def\g{\mathfrak g}

\def\R{{\mathfrak R}}

\def\begeq{\begin{equation}}
	\def\endeq{\end{equation}}

\def\R{\Bbb R}

\def\g{\gamma}

\allowdisplaybreaks

\begin{document}
	
	\title[Non-degeneracy of double-tower solutions for  Schr\"odinger equation
	]
	{   Non-degeneracy of double-tower solutions for nonlinear Schr\"odinger equation and  applications
	}

	\author{Yuan Gao and  Yuxia Guo
	}
	
	\address{Yuan Gao,
		\newline\indent Department of Mathematical Sciences, Tsinghua University,
		\newline\indent Beijing 100084,  P. R. China.
	}
	\email{gaoy22@mails.tsinghua.edu.cn}

	\address{Yuxia Guo,
		\newline\indent Department of Mathematical Sciences, Tsinghua University,
		\newline\indent Beijing 100084,  P. R. China.
	}
	\email{yguo@tsinghua.edu.cn}
	
\thanks{This work  is supported by NSFC (No. 12031015, 12271283).}

	\begin{abstract}
		
		This paper is concerned with  the following  nonlinear  Schr\"odinger equation
		\begin{equation}
			\label{eq}
			- \Delta u + V(|y|)u=u^{p},\quad u>0 \ \   \mbox{in} \  \R^N, \ \ \
			u \in H^1(\R^N),
		\end{equation}
		where $V(|y|)$ is a  positive function, $1<p <\frac{N+2}{N-2}$. Based on the local Pohozaev identities and blow-up analysis, we first prove a non-degeneracy result for double-tower solutions constructed in \cite{DM} in a suitable symmetric space. As an application, we obtain the existence of new type solutions for \eqref{eq}.
		
		\vspace{2mm}
		
		
		\vspace{2mm}
		

	\end{abstract}
	
	\maketitle
	
	\section{Introduction}

	We consider the following nonlinear Schr\"odinger equations:
	\begin{equation}\label{equation0}
		\begin{cases}
			-\Delta u + V(y) u = |u|^{p-1} u, \quad & \text{in}\; \mathbb R^N,\\
			\lim_{|y|\to+\infty}
			u(y)=0,
		\end{cases}
	\end{equation}
or the following nonlinear field equations in the subcritical case:

\begin{equation}\label{equation1}
		\begin{cases}
			-\Delta u +  u = Q(x)|u|^{p-1}u, \quad  & \text{in}\; \mathbb R^N,\\
			\lim_{|y|\to+\infty}
			u(y)=0,
		\end{cases}
	\end{equation}
	with  $V(y),Q(y)$  satisfying some suitable conditions.
 These problems originate from seeking the standing waves solutions with the  form of
	$
	\psi(y,t) = e^{{  \bf{ i } } \lambda t } u(y)
	$
	to the time-dependent Schr\"odinger equation
	\begin{align}\label{original1}
		-{  \bf{ i } } \frac{\partial \psi}{\partial t}  = \Delta \psi - \widetilde V(y) \psi + |\psi|^{p-1} \psi,     \quad\text{for} ~ (y,t) \in \mathbb{R}^N\times \mathbb{R},
	\end{align}
	where ${\bf{i}}$ is the imagine,  $ p>1$ and $\widetilde V(y)=V(y)-\lambda$. Due to their widely research background in physical sciences such as in nonlinear optics, plasma physics, quantum mechanics and condensed matter, the problem of finding  the existence of nontrivial solutions for \eqref{equation0} or \eqref{equation1} has received a lot of attention in last decades. We refer the readers to \cite{BL,BLi,C,CNY1,CNY2,CDS,DN,Lion} and the references therein for comprehensive results. We also refer the readers to
	\cite{ABC, AoM, AMN1, AMN2, CNY1,CNY2,DY,DF1,DF2,DF3,DF4,DK,DM,FW,KW,Li,LinNW,MPW,NY,R,WY,WY1} for the singular perturbed problems related to the problem \eqref{equation0} or \eqref{equation1}.

	In this paper, we restrict our attention to  the problem \eqref{equation0} under the assumption that the potential function $V(y)>0$ is bounded and radial, which means that $V(y)=V(|y|),$ that is the problem:
	\begin{equation}\label{equation}
		\begin{cases}
			-\Delta u + V(|y|) u = u^p, \quad u>0 & \text{in}\; \mathbb R^N,\\
			u(y) \in H^1(\R^N),
		\end{cases}
	\end{equation}
	where $N\ge 3$ and $p\in \bigg(1, \displaystyle\frac{N+2}{N-2}\bigg)$.
	
If $V(|y|)\equiv 1$, it is well known (see \cite{K,NT}) that \eqref{equation} has a unique  positive solution $U(y)$, which is radially symmetric and
	\begin{equation}
		\lim_{|y| \rightarrow +\infty}  U(|y|)  e^{|y|} y^{\frac{N-1}{2}} = C_0<+\infty,  \quad U'(r)<0    \quad\text{and } ~ \lim_{|y| \rightarrow +\infty} \frac{U(y)}{U'(y)}=-1,
	\end{equation}
	where $C_0$ is a constant depends on $N$ and $p$.
	Moreover,  $U(|y|)$ is non-degenerate in the sense that the kernel of the linear operator  $-\Delta + \mathbb{I} -
	pU^{p-1} $  in $H^1(\mathbb{R}^N) $ is spanned by  $$ \bigg\{ \displaystyle\frac{\partial U}{\partial {y_1}  }, \cdots, \displaystyle\frac{\partial U}{\partial {y_N} } \bigg\}.$$

We use the notation:   $U_{x_j}(y)=U(y-x_j).$  Under the assumptions that $V(y)=V(|y|)$ satisfies:

({\bf V}):  There  are constants $a>0$, $\alpha>1$,  and $\beta>0$,
	such that
	
	\begin{equation}\label{V}
		V(r)= 1+\frac a {r^\alpha} +O\bigl(\frac1{r^{\alpha+\beta}}\bigr),
	\end{equation}
	as $r\to +\infty$.
	 Wei and Yan \cite{WY} first constructed infinitely many non-radial positive solutions for \eqref{equation}, which we called bubble solutions, by gluing the bubbles $U_{x_j}$ at the circle of $(x_1,x_2)$-plane:
	\begin{equation}
		x_j=r\Bigl(\cos\frac{2(j-1)\pi}k, \sin\frac{2(j-1)\pi}k,{\bf 0}\Bigr),\quad j=1,\cdots,k,
	\end{equation}
	where ${\bf 0}$ is the zero vector in $\R^{N-2}, r \sim k\ln k,$ as $k \to +\infty.$ This construction is usually called a desingularization of equatorial (see \cite{DM}). Based on this idea, Duan and Musso \cite{DM} consider the existence of the bilayer of bubble solutions, which we called double-tower solutions. This kind of solutions are centered at the points lying on the top and the bottom circles of a cylinder.  More precisely, set $k$ be an integer number and
	\begin{equation*}
		x_{j}^{+}=r\Bigl(\sqrt{1-h^2}\cos\frac{2(j-1)\pi}k, \sqrt{1-h^2}\sin\frac{2(j-1)\pi}k,h,{\bf 0}\Bigr),\quad j=1,\cdots,k,
	\end{equation*}
	and
	\begin{equation*}
		x_{j}^{-}=r\Bigl(\sqrt{1-h^2}\cos\frac{2(j-1)\pi}k, \sqrt{1-h^2}\sin\frac{2(j-1)\pi}k,-h,{\bf 0}\Bigr),\quad j=1,\cdots,k,
	\end{equation*}
	where ${\bf 0}$ is the zero vector in $\R^{N-3}$, $(r,h)\in [r_0k\ln k, r_1k\ln k]\times[\displaystyle\frac{h_0}{k},\displaystyle\frac{h_1}{k}]$ with some constants $r_0,r_1,h_0,h_1>0$.

For any point $y \in \mathbb{R}^{N}$, we set $y = (y',y''),y'\in\mathbb{R}^{3},y''\in\mathbb{R}^{N-3}$. Define
	\[
	\begin{split}
		H_{s}=\bigg\{ u:    u\;\text{is even in} \;&y_2,y_4,y_5,\cdots, y_N,\quad u\left(\sqrt{y_1^2+y_2^2}\cos\theta , \sqrt{y_1^2+y_2^2}\sin\theta,y_3, y''\right)\\
		& =u\left(\sqrt{y_1^2+y_2^2}\cos\left(\theta+\frac{2\pi j}k\right) , \sqrt{y_1^2+y_2^2}\sin\left(\theta+\frac{2\pi j}k\right), y_3,y''\right)
		\bigg\}.
	\end{split}
	\]

	Let
	\[
	W_{r,h}(y)=\sum_{j=1}^k \left(U_{x_j^+}(y)+U_{x_j^-}(y)\right).
	\]
	In \cite{DM}, the following condition is assumed.

	$({\bf V_0})$:  There  are constants $a_0>0$, $m>\displaystyle \max \bigg\{ \frac{4}{p-1},2\bigg\}$,  and $\gamma>0$,
	such that
	
	\begin{equation}\label{V_0}
		V(r)= 1+\frac {a_0} {r^m} +O\bigl(\frac1{r^{m+\gamma}}\bigr),
	\end{equation}
	as $r\to +\infty$. Under this assumption, the double-tower solutions constructed in \cite{DM} are stated as follows:

	\begin{thmA}  \label{thmA}
		Suppose that $V(|y|)$ satisfies $({\bf V_0})$. Then there is an
		integer $k_0>0$, such that for any integer  $k\ge k_0$, \eqref{equation}
		has a solution $u_k$ of the form
		\[
		u_k = W_{r_k,h_k}(y)+\omega_k,
		\]
		where  $\omega_k\in H_s\cap  H^1(\R^N)$,
		\begin{equation}
			(r_k,h_k) \in \mathscr{S}_k:= \bigg [\displaystyle{\frac{m-\alpha_1}{2\pi}}k\ln k,\displaystyle{\frac{m+\alpha_1}{2\pi}}k\ln k\bigg]\times \bigg[\displaystyle{\frac{\pi(m+2)-\alpha_2}{m}} \displaystyle{\frac{1}{k}},\displaystyle{\frac{\pi(m+2)+\alpha_2}{m}} \displaystyle{\frac{1}{k}} \bigg],
		\end{equation}
		and as $k\to +\infty$,
		\[
		\int_{\R^N} \bigl(|D \omega_k |^2+\omega_k^2\bigr)\to 0.
		\]
	\end{thmA}
	
	Of course, by suing the similar arguments, we can construct similar double-tower solutions  by gluing two layers $n$ bubbles $U_{x_i}$ at the circle of $(x_3,x_4)$-plane with height $\tilde{h}$:
\begin{equation*}
		x_{i}^{+}=\tilde{r}\Bigl(0, 0, \sqrt{1-\tilde{h}^2}\cos\frac{2(i-1)\pi}n, \sqrt{1-\tilde{h}^2}\sin\frac{2(i-1)\pi}n, \tilde{h}, {\bf 0}\Bigr),\quad i=1,\cdots,n,
	\end{equation*}
	and
	\begin{equation*}
		x_{i}^{-}=\tilde{r}\Bigl(0, 0, \sqrt{1-\tilde{h}^2}\cos\frac{2(i-1)\pi}n, \sqrt{1-\tilde{h}^2}\sin\frac{2(i-1)\pi}n,-\tilde{h},{\bf 0}\Bigr),\quad i=1,\cdots,n,
	\end{equation*}
	where ${\bf 0}$ is the zero vector in $\R^{N-5}$, $(\tilde{r},\tilde{h})\in [{\tilde{r}_0}n\ln n,{\tilde{r}_1}n\ln n]\times[\displaystyle\frac{\tilde{h}_0}{n},\displaystyle\frac{\tilde{h}_1}{n}]$ with some constants $\tilde{r}_0,\tilde{r}_1,\tilde{h}_0,\tilde{h}_1>0$.
Our question is can we obtain new type solutions for \eqref{equation} whose main order is like the sum of the two double-tower solutions which  located at the cylinder in  the $(x_1, x_2) $-plane and the $(x_3, x_4) $-plane separately? By carefully analysis, we find that if the number of bubbles, that is $n, k$ are different orders, it is very hard to finish the construction process as before. Motivated by the work  Guo, Musso, Peng and Yan \cite{GMPY1}, instead of the construction directly, by  using  local Pohozaev identities,  we first prove  a non-degeneracy result for  the so called double-tower solutions constructed in \cite{DM}, then we proceed a gluing argument to obtain the existence of  new type solutions.  We would like to mention that similar idea and method  are used  in \cite{GH,GHLN,GMPY} for the study of the prescribed scalar curvature problem. However, compared with the prescribed scalar curvature equation,  the study of non-degeneracy for nonlinear Schr\"odinger equations is quite difficult, for the reason that the critical point of $V(|y|)$ is located at infinity, and the decay of the solution $U(|y|)$ is exponential rather than polynomial.  In fact, non-degeneracy result is quite important in further study of the existence and other related problems, we refer the reader to \cite{DMW,GGH,MM} and references therein for any other interesting  related results.
	
	
	Before the statement of our theorems, we make the following assumptions on the potential function $V(|y|)$:

		$({\bf V_1}): V(r)= 1+\displaystyle\frac {a_1} {r^m} +\displaystyle\frac {a_2} {r^{m+1}} +O\bigl(\frac1{r^{m+2}}\bigr),$

		$({\bf V_2}): V'(r)= -\displaystyle\frac {a_1m} {r^{m+1}} -\displaystyle\frac {a_2(m+1)} {r^{m+2}} +O\bigl(\frac1{r^{m+3}}\bigr),$
where  $m>\max\bigg\{\displaystyle\frac 4{p-1}, 4\bigg\}$, $a_1>0$  and $a_2$ are some   constants. 

Meanwhile,
	for a function  $u(y) \in H_s\cap  C(\R^N)$, we introduce  the following norm:
	\begin{equation}\label{norm}
		\| u\|_*= \sup_{y\in \mathbb R^N} \frac{|u(y)|}{\sum_{j=1}^k \left(e^{-\tau |y-x_j^{+}|}+e^{-\tau |y-x_j^{-}| }\right)},
	\end{equation}
	where $\tau$ is any fixed small constant.
	
	\medskip
	
	Here are the  main results of our paper.

	\begin{theorem}\label{th11}
		
		Suppose that $V(|y|)$ satisfies $({\bf{V_1}})$  and  $({\bf{V_2}})$, then there exist a large constant $k_0>0$, such that for any integer $k>k_0$, the double-tower solution $u_k$ obtained in ${ Theorem~A}$ is non-degenerate in the sense that if $\xi\in H_s\cap H^1(\mathbb R^N)$ is a solution of the following equation:
		\begin{equation}\label{2-22-12}
			L_k \xi:= -\Delta \xi + V(|y|) \xi -p u_k^{p-1}\xi  \quad \text{in} \ \R^N,
		\end{equation}
		then $\xi=0$.
		
	\end{theorem}

	As an application of  Theorem~\ref{th11}, we have the following new  existence  result for \eqref{equation}. Assume that $N\ge 6$,  for any large integer $n>0$, define
	
	\begin{equation*}
		p_{j}^{+}=t\Bigl(0, 0, 0,\sqrt{1-l^2}\cos\frac{2(j-1)\pi}n, \sqrt{1-l^2}\sin\frac{2(j-1)\pi}n,l,{\bf 0}\Bigr),\quad j=1,\cdots,n,
	\end{equation*}
	and
	\begin{equation*}
		p_{j}^{-}=t\Bigl(0, 0, 0,\sqrt{1-l^2}\cos\frac{2(j-1)\pi}n, \sqrt{1-l^2}\sin\frac{2(j-1)\pi}n,-l,{\bf 0}\Bigr),\quad j=1,\cdots,n,
	\end{equation*}
	where ${\bf 0}$ is the zero vector in $\R^{N-6}, (t,l)\in [t_0n\ln n, t_1n\ln n]\times[\displaystyle\frac{l_0}{n},\displaystyle\frac{l_1}{n}]$ with some constants $t_0,t_1,l_0,l_1>0$. Let
	\[
	\begin{split}
		X_s=\Bigg\{ u:  u\in H_s, u\;& \text{is even in} \;y_h, h=1,\cdots,N, \ u\left(y_1, y_2,y_3, \sqrt{y_4^2+y_5^2}\cos\theta , \sqrt{y_4^2+y_5^2}\sin\theta,y_6, y^*\right)\\
		& =
		u\left(y_1,y_2,y_3,\sqrt{y_4^2+y_5^2}\cos\left(\theta+\frac{2\pi j}n\right) , \sqrt{y_4^2+y_5^2}\sin\left(\theta+\frac{2\pi j}n\right),y_6, y^*\right)
		\Bigg\}.
	\end{split}
	\]
	Here $y^*=(y_7,\cdots,y_N) \in \R^{N-6}.$ Then we have

	\begin{theorem}\label{th1.2}
		
		Suppose that $N\ge 6$ and  $V(|y|)$ satisfies the assumptions in {Theorem~\ref{th11}}. Let $u_k$ be a solution in { Theorem~A} and $k>0$ is a fixed large even number. Then there is an
		integer $n_0>0$, depending on $k$, such that for any even number  $n\ge n_0$,  \eqref{equation}
		has a solution of the form
		\begin{equation}\label{sol}
			u_{k,n}= u_k +\sum_{j=1}^n \left(U_{p_j^{+}}+U_{p_j^{-}}\right)+\xi_n,
		\end{equation}
		where  $\xi_n\in X_s\cap  H^1(\R^N),$
		\begin{equation}
			(t,l) \in \mathscr{T}_k:= \bigg [\displaystyle{\frac{m-\beta_1}{2\pi}}n\ln n,\displaystyle{\frac{m+\beta_1}{2\pi}}n\ln n\bigg]\times \bigg[\displaystyle{\frac{\pi(m+2)-\beta_2}{m}} \displaystyle{\frac{1}{n}},\displaystyle{\frac{\pi(m+2)+\beta_2}{m}} \displaystyle{\frac{1}{n}} \bigg],
		\end{equation}
		and as $n\to +\infty$,
		\[
		\int_{\R^N} \bigl(|D \xi_k |^2+\xi_k^2\bigr)\to 0.
		\]
	\end{theorem}
	\medskip
	
	\begin{remark} Note that in the construction of Theorem A, the correction term is goes to zero in the sense that
\[
		\int_{\R^N} \bigl(|D \omega_k |^2+\omega_k^2\bigr)\to 0.
		\]
In order to prove its non-degeneracy, we need the point-wise estimate. So we have to revisit the proof of Theorem A such that the correction term goes to zero  under the norm  $||\cdot ||_*.$ 
\end{remark}

	\begin{remark}
		We would like to mention that, in general, the idea and method we developed here can be applied to construct continuously newer solutions to \eqref{equation}. For example,  we can prove the non-degeneracy of new double-tower solutions $u_{k,n}$ which are constructed in {Theorem \ref{th1.2}} and apply it as $N\ge 9$, to seek for newer solutions of the form
		\begin{equation}
			u\approx u_{k,n} +\sum_{j=1}^n \left(U_{q_j^{+}}+U_{q_j^{-}}\right),
		\end{equation}
		with the points $q_j^{+}$ and $q_j^{-}$ are defined in $\R^N$ as
		\begin{equation*}
			q_{j}^{+}=t'\Bigl(0, 0, 0,0,0,0,\sqrt{1-l'^2}\cos\frac{2(j-1)\pi}n, \sqrt{1-l'^2}\sin\frac{2(j-1)\pi}n,l',{\bf 0}\Bigr),\quad j=1,\cdots,n',
		\end{equation*}
		\begin{equation*}
			q_{j}^{-}=t'\Bigl(0, 0, 0,0,0,0,\sqrt{1-l'^2}\cos\frac{2(j-1)\pi}n, \sqrt{1-l'^2}\sin\frac{2(j-1)\pi}n,-l',{\bf 0}\Bigr),\quad j=1,\cdots,n',
		\end{equation*}
		where ${\bf 0}$ is the zero vector in $\R^{N-9}$. And so on.
	\end{remark}
	
	The paper is organized as follows. In Section~2, we will revisit the proof of { Theorem~A} in order to obtain a point-wise estimate for the correction function under the norm we defined in \eqref{norm}. Such estimate is necessary in the proof of { Theorem~\ref{th11}}. Also, we give the energy expansion of the solution, which is essential in constructing the new solutions in Section~4. In Section~3, we using the local Pohozaev identities to obtain the non-degeneracy of the solutions. Section~4 is devoted to the construction of new type  solutions to \eqref{equation}. We put the establishment of the local Pohozaev identities and some basic estimates in Appendices \ref{appendixA}-\ref{appendixB}.

	\section{Revisit the existence of double-tower solutions}\label{sec2}
	To get the non-degeneracy of double-tower solutions, we need point-wise estimate for the correction function. So we first  briefly revisit the existence of double-tower solutions for problem \eqref{equation} under the norm in \eqref{norm}, and then we calculate the expansion for the energy functional. We use the same notations as in Section 1.
	
	We consider the linearized problem:
	\begin{align}\label{lin}
		\begin{cases}
			-\Delta {\phi}+V(|y|)\phi-p{W_{r,h}^{p-1}}\phi =f+\sum\limits_{j=1}^k\sum\limits_{\ell=1}^2 c_{k\ell} \Big( U_{x^{+}_{j}}^{p-1}\overline{\mathbb{Z}}_{\ell j}+ U_{x^{-}_{j}}^{p-1}\underline{\mathbb{Z}}_{\ell j}\Big)
			\;\;
			\text{in}\;
			\R^N,
			\\[2mm]
			\phi\in \mathbb{E}_k,
		\end{cases}
	\end{align}
	for some constants $c_{k\ell}, \ell=1,2 $, where the functions $\overline{\mathbb{Z}}_{\ell j}$ and $\underline{\mathbb{Z}}_{\ell j}$ are given by
	\begin{align*}
		\overline{\mathbb{Z}}_{1j}=\frac{\partial U_{x^{+}_{j}}}{\partial r},
		\qquad
		\underline{\mathbb{Z}}_{1j}=\frac{\partial U_{x^{-}_{j}}}{\partial r},
		\qquad
		\overline{\mathbb{Z}}_{2j}=\frac{\partial U_{x^{+}_{j}}}{\partial h},
		\qquad
		\underline{\mathbb{Z}}_{2j}=\frac{\partial U_{x^{-}_{j}}}{\partial h}.
	\end{align*}
	for $ j= 1, \cdots, k $. Moreover the function $\phi$ belongs to the set $\mathbb{E}_k$  which is given by
	\begin{align}\label{SpaceE}
		\mathbb{E}_k= \Big\{\phi:  \phi\in H_s \cap {C(\R^N)}, \quad \big<U_{x^{+}_{j}}^{p-1}  \overline{\mathbb{Z}}_{\ell j} ,\phi \big> = \big<U_{x^{-}_{j}}^{p-1}  \underline{\mathbb{Z}}_{\ell j} ,\phi \big> = 0,  \ j= 1, \cdots, k, \ \ell= 1,2\Big\},
	\end{align}
	where $\big<u,v\big> = \displaystyle\int_{\R^N}{u v} .$
	\medskip
	
	\begin{lemma}\label{lem2.1}
		Suppose that $\phi_{k}$ solves \eqref{lin} for $f=f_{k}$. If $\| f_k\|_* \to 0$  as $k\to +\infty$, so does  $\|\phi_{k}\|_{{*}}$.
	\end{lemma}
	\begin{proof}
		
		We prove it by contradiction. Suppose that there exists a sequence of  $({r_k}, {h_k})\in {{\mathscr S}_k}$, and $(\phi_k,c_{k\ell}),\ \ell=1,2$ solving \eqref{lin} with  $f=f_k,r= {r_k}, h= {h_k}$, such that  $\Vert f _{k}\Vert_{{*}}\to 0 $, and  $\|\phi_k\|_{{*}}\ge c'>0$ as $k \to + \infty$. Without loss of generality, we may assume that  $\|\phi_k\|_{{*}}=1$. For simplicity, we drop the subscript $k$.
		
		Using the Green's function $G(y, x)$ of $-\Delta+V(|y|)$ in $\mathbb{R}^N$, we have
		
		\begin{equation}\label{2-l21}
			\phi(x) =\int_{\mathbb R^N} G(y, x) \bigl(p W_{r,h}^{p-1} \phi +f +\sum\limits_{j=1}^k\sum\limits_{\ell=1}^2 c_{k\ell} \bigg( U_{x^{+}_{j}}^{p-1}\overline{\mathbb{Z}}_{\ell j}+ U_{x^{-}_{j}}^{p-1}\underline{\mathbb{Z}}_{\ell j}\Big)\bigg),
		\end{equation}
		and
		\begin{equation*}
			0< G(y, x) \le  C e^{-|y-x|},
		\end{equation*}
		as $|y-x| \to +\infty,$
		\begin{equation*}
			0< G(y, x) \le  \frac{C}{|y-x|^{N-2}},
		\end{equation*}
		as $|y-x| \to 0.$
		
		By  calculation, we have for $\tau>0$ small,
		\begin{align*}
			&\bigg| \int_{\mathbb R^N} G(y, x) \big( pW_{r,h}^{p-1} \phi +f\big)\bigg|\\
			\le& C \|\phi\|_*  \int_{\mathbb R^N} G(y, x) \bigg( \sum_{j=1}^k  \big(e^{-|y-{x^{+}_{j}}|}+e^{-|y-{x^{-}_{j}}|}\big) \bigg)^{p-1}\bigg( \sum_{j=1}^k  \big(e^{-\tau|y-{x^{+}_{j}}|}+e^{-\tau|y-{x^{-}_{j}}|}\big)\bigg)\\
			&+C\|f\|_* \bigg(\sum_{j=1}^k \int_{\mathbb R^N} G(y, x) \big(e^{ -\tau|y-{x^{+}_{j}}|}+e^{ -\tau|y-{x^{-}_{j}}|}\big)\bigg)\ \\
			\le& C\bigg(
			\|\phi\|_* \sum_{j=1}^k   \big(e^{-2\tau|x-{x^{+}_{j}}|}+e^{-2\tau|x-{x^{-}_{j}}|}\big)+\|f\|_* \sum_{j=1}^k  \big( e^{{-\tau|x-{x^{+}_{j}}|}}+e^{{-\tau|x-{x^{-}_{j}}|}}\big)\bigg).
		\end{align*}
		Moreover,
		\begin{align*}
			\bigg| \int_{\mathbb R^N} G(y, x)  \sum_{j=1}^k  \bigg( U_{x^{+}_{j}}^{p-1}\overline{\mathbb{Z}}_{1 j}+ U_{x^{-}_{j}}^{p-1}\underline{\mathbb{Z}}_{1 j}\Big)
			\bigg|\le C \sum_{j=1}^k  \big(e^{-|x-{x^{+}_{j}}|}+e^{-|x-{x^{-}_{j}}|}\big).
		\end{align*}
		and
		\begin{align*}
			\bigg| \int_{\mathbb R^N} G(y, x)  \sum_{j=1}^k  \bigg( U_{x^{+}_{j}}^{p-1}\overline{\mathbb{Z}}_{2 j}+ U_{x^{-}_{j}}^{p-1}\underline{\mathbb{Z}}_{2 j}\Big)
			\bigg|\le C r \sum_{j=1}^k  \big(e^{-|x-{x^{+}_{j}}|}+e^{-|x-{x^{-}_{j}}|}\big).
		\end{align*}
		That is,
		\begin{align}
			\bigg| \int_{\mathbb R^N} G(y, x)  \sum_{j=1}^k  \bigg( U_{x^{+}_{j}}^{p-1}\overline{\mathbb{Z}}_{\ell j}+ U_{x^{-}_{j}}^{p-1}\underline{\mathbb{Z}}_{\ell j}\Big)
			\bigg|\le C (1+\delta_{\ell 2}r) \sum_{j=1}^k  \big(e^{-|x-{x^{+}_{j}}|}+e^{-|x-{x^{-}_{j}}|}\big).
		\end{align}
		
		On the other hand, according to
		\begin{equation}\label{equationofcl}
			\bigg<-\Delta {\phi}+V(|y|)\phi -p
			W_{r,h}^{p-1}{\phi} , \overline{\mathbb{Z}}_{q 1}
			\bigg> = \bigg<f + \sum_{j=1}^k\sum_{\ell=1}^2 \Big(\, {c_\ell}U_{x^{+}_{j}}^{p-1}\overline{\mathbb{Z}}_{\ell j}+{c_\ell}U_{x^{-}_{j}}^{p-1}\underline{\mathbb{Z}}_{\ell j}\,\Big), \overline{\mathbb{Z}}_{q 1}\bigg>.
		\end{equation}
		We can prove that $c_\ell\to 0$ as $k \to \infty$. In fact, we have
		\begin{align*}
			\bigg|\bigg<&-\Delta {\phi}+V(|y|)\phi -pW_{r,h}^{p-1}{\phi} , \overline{\mathbb{Z}}_{q 1}\bigg>\bigg|\nonumber\\
			&\le C \|\phi\|_*\int_{\R^N}{\sum_{j=1}^k \big(e^{ -\tau|y-{x^{+}_{j}}|}+e^{ -\tau|y-{x^{-}_{j}}|}\big)}\left((V(|y|)-1)\overline{\mathbb{Z}}_{q 1}-p \bigg(W_{r,h}^{p-1}- U_{x^{+}_{1}}^{p-1} \bigg)\overline{\mathbb{Z}}_{q 1} \right)\nonumber\\
			&\le C (1+\delta_{q 2}r) \|\phi\|_* \bigg(\frac{1}{ r^m }+e^{-\frac{(1-\tau)r}{2}}+ \sum_{j=1}^k e^{-\tau|{x^{+}_{1}-{x^{+}_{j}}|} }  + e^{-\tau|{x^{+}_{1}-{x^{-}_{1}}|} }
			\bigg)\\
&\quad\times\left(\sum_{j=1}^k e^{-\tau|{x^{+}_{1}-{x^{+}_{j}}|} }  + e^{-\tau|{x^{+}_{1}-{x^{-}_{1}}|} }  \right)\\
			&\le C (1+\delta_{q 2}r) \|\phi\|_* \frac{1}{k^{\tau m}}
			\left(\sum_{j=1}^k e^{-\tau|{x^{+}_{1}-{x^{+}_{j}}|} }  + e^{-\tau|{x^{+}_{1}-{x^{-}_{1}}|} }  \right).
		\end{align*}
		And
		\begin{align*}
			\bigg|\bigg<&f , \overline{\mathbb{Z}}_{q 1}\bigg>\bigg|\le C  (1+\delta_{q 2}r) \|f\|_* \left(\sum_{j=1}^k e^{-\tau|{x^{+}_{1}-{x^{+}_{j}}|} }  + e^{-\tau|{x^{+}_{1}-{x^{-}_{1}}|} }  \right).
		\end{align*}
		Noting that
		\begin{eqnarray*}
			\big< U_{x^{+}_{1}}^{p-1}\overline{\mathbb{Z}}_{\ell 1}  , \overline{\mathbb{Z}}_{q 1} \big> =\left\{\begin{array}{lll} o_k(1),\qquad &\text{if}\quad \ell \neq q,
				\\[2mm]
				\space\\[2mm]
				\bar{c} (1+\delta_{q2} r^2),\qquad &\text{if} \quad\ell=q. \end{array}\right.
		\end{eqnarray*}
		Then we have
		
		\begin{align} \label{estimate cl}
			|c_{\ell}|=  \left(\frac{1+r\delta_{\ell2}}{1+r^2\delta_{\ell2} } +o_k(1)\right) O_k\Bigl(\frac 1{{k}^{\tau m}}\|{\phi}\|_{{*}}+\|f\|_{{*}} \Bigr)= o_k\left(\frac{1}{1+r}\right), \quad {\mbox {as}} \quad k \to \infty.
		\end{align}
		
		Therefore, it holds that
		
		\begin{align}\label{100-27-2}
			&\quad |\phi(x)|\bigg( \sum_{j=1}^k  \big(e^{-\tau|x-{x^{+}_{j}}|}+e^{-\tau|x-{x^{-}_{j}}|}\big)\bigg)^{-1}  \nonumber \\
			&\le (C\|f \|_* + o_k(1)) + C\|\phi \|_* \bigg(\sum_{j=1}^k \big(e^{-\tau|x-{x^{+}_{j}}|}+e^{-\tau|x-{x^{-}_{j}}|}\big)\bigg)^{-1}\sum_{j=1}^k   \big(e^{-2\tau|x-{x^{+}_{j}}|}+e^{-2\tau|x-{x^{-}_{j}}|}\big).
		\end{align}
		
		Since $\phi\in \mathbb{E}_k$, we have $\| \phi\|_{L^\infty(B_R(x^{+}_{j}))}\to 0$  for any $R>0$. Thus, from \eqref{100-27-2} we get that $\|\phi\|_*\to 0$. This is a contradiction.
		
	\end{proof}

	We can directly derive from Lemma \ref{lem2.1} that linearized problem \eqref{lin} has the following existence, uniqueness results:
	\begin{proposition}\label{pro2.2}
		There exist  $k_0>0  $ and a constant  $C>0 $ such
		that for all  $k\ge k_0 $ and all  $f_k\in L^{\infty}(\R^N) $, problem
		$(\ref{lin}) $ has a unique solution  $\phi_k\equiv {\bf {\Phi}}_k(f_k) $. Besides,
		\begin{equation}\label{Le}
			\Vert \phi_k\Vert_{{*}}\leq C\|f_k\|_{{*}},\qquad
			|c_{k\ell}|\leq  \frac{C} {1+\delta_{\ell 2} r}\|f_k\|_{{*}}, \quad \ell=1,2.
		\end{equation}
	\end{proposition}

	Now, we rewrite problem \eqref{lin} as
	\begin{align}\label{linn}
		\begin{cases}
			L \omega_k =  l_k  +  R(\omega_k)+ \sum\limits_{j=1}^k\sum\limits_{\ell=1}^2 c_{k\ell} \Big( U_{x^{+}_{j}}^{p-1}\overline{\mathbb{Z}}_{\ell j}+ U_{x^{-}_{j}}^{p-1}\underline{\mathbb{Z}}_{\ell j}\Big),\\[2mm]
			\omega_k \in  \mathbb{E}_k,
		\end{cases}
	\end{align}
	where
	\begin{equation*}
		L \omega_k := -\Delta \omega_k+V(|y|) \omega_k-pW^{p-1}_{r,h}\omega_k,
	\end{equation*}
	
	\begin{equation*}
		l_k:=-\sum_{j=1}^k (V(|y|)-1) (U_{x^{+}_{j}}+U_{x^{-}_{j}}) +\big( W_{r,h}^{p}-\sum_{j=1}^k \big(U_{x^{+}_{j}}^p+U_{x^{-}_{j}}^p\big)\big),
	\end{equation*}
	and
	\begin{equation*}
		R(\omega_k):=(W_{r,h}+\omega_k)^p  -W_{r,h}^p  -pW_{r,h}^{p-1} \omega_k.
	\end{equation*}
	
	For $l_k$, we have the following estimate.
	\begin{lemma}\label{lem2.3}
		Suppose that  $N\ge 3$, and $ V(|y|)$ satisfies $({\bf V_1})$, $(r,h) \in{{\mathscr S}_k}$. There exists an integer  $k_0>0$ large enough, such that for any $k \ge k_0, $
		\[
		\|l_k\|_* \le  \frac{C }{k^{\min \{ \frac p2-\tau, 1 \}m}}.
		\]
		
	\end{lemma}
	
	\begin{proof}
		
		For any $y\in \Omega_1^{+}$, noting that if $y \le \displaystyle\frac{1}{2}|{x^{+}_{1}}|$,
		\begin{equation*}
			\frac{C}{1+  |y|^m} e^{-|y-{x^{+}_{1}}|(1-\tau)}\le \frac{C}{ |{x^{+}_{1}}|^m}.
		\end{equation*}
		Else if $y > \displaystyle\frac{1}{2}|{x^{+}_{1}}|$,
		\begin{equation*}
			\frac{C}{1+  |y|^m} e^{-|y-{x^{+}_{1}}|(1-\tau)}\le C e^{-|{x^{+}_{1}}|(\frac{1}{2}-\tau)} \le \frac{C}{ |{x^{+}_{1}}|^m}.
		\end{equation*}
		Thus, we can give the estimate of the first term of $l_k$:
		\[
\begin{array}{ll}
		\bigg| \sum_{i=1}^k (V(|y|)-1) (U_{x^{+}_{i}}+  U_{x^{-}_{i}})\bigg|&\displaystyle\le  \frac{C}{1+  |y|^m} \sum_{i=1}^k \big(e^{-|y-{x^{+}_{i}}|}+ e^{-|y-{x^{-}_{i}}|} \big)\\&\displaystyle\le  \frac{C}{   r^m} \sum_{i=1}^k \big(e^{-\tau |y-{x^{+}_{i}}|}+  e^{-\tau |y-{x^{-}_{i}}|} \big).\end{array}
		\]
		
		Meanwhile, for the second term of $l_k$, if $p>2$, it holds
		
		\[
		\begin{split}
			\bigg| W_{r,h}^{p}-\sum_{i=1}^k \big(U_{x^{+}_{i}}^p+U_{x^{-}_{i}}^p\big)\bigg|\le & C \left(\sum_{i=2}^k U_{x^{+}_{1}}^{p-1} U_{x^{+}_{i}}+U_{x^{+}_{1}}^{p-1} U_{x^{-}_{1}}\right)\\ \le &C e^{-(p-2) |y-{x^{+}_{1}}|}  \left(\sum_{i=2}^k e^{- |{x^{+}_{i}}-{x^{+}_{1}}|}+ e^{- |{x^{+}_{1}}-{x^{-}_{1}}|}\right)\\
			\le & C e^{-\tau |y-{x^{+}_{1}}|}  \left(\sum_{i=2}^k e^{- |{x^{+}_{i}}-{x^{+}_{1}}|}+ e^{- |{x^{+}_{1}}-{x^{-}_{1}}|}\right),
		\end{split}
		\]
		while if $p\in (1, 2]$,
		\[
		\begin{split}
			\bigg| W_{r,h}^{p}-\sum_{i=1}^k \big(U_{x^{+}_{i}}^p+U_{x^{-}_{i}}^p\big)\bigg|\le& C \sum_{i=2}^k U_{x^{+}_{1}}^{\frac p2} U_{x^{+}_{i}}^{\frac p2}+U_{x^{+}_{1}}^{\frac p2} U_{x^{-}_{1}}^{\frac p2}\\ \le & C e^{-\tau  |y-{x^{+}_{1}}|}  \left(\sum_{i=2}^k e^{-(\frac p2-\tau) |{x^{+}_{i}}-{x^{+}_{1}}|}+ e^{-(\frac p2-\tau) |{x^{+}_{1}}-{x^{-}_{1}}|}\right).
		\end{split}
		\]
		
		Noting that $(r,h) \in{{\mathscr S}_k}$, thus,
		
		$$|x^{+}_2-x^{+}_1|=2r(1-h^2)^{\frac{1}{2}} \sin (\frac{\pi}{k}) = \frac{2\pi r}{k} (1+o(1)) =m \ln k (1+o(1))$$
		and
		
		$$|x^{+}_1-x^{-}_1|=2rh = (m+2) \ln k (1+o(1)),$$
		which leads to
		\begin{align*}
			\|l_k\|_* \le & \frac{C }{r^m}+ Ce^{- \min \{ \frac p2-\tau, 1 \} (|x^{+}_2-x^{+}_1|+|x^{+}_1-x^{-}_1|)}\\
			\le & \frac{C }{r^m}+ \frac{C }{k^{\min \{ \frac p2-\tau, 1 \}m}} \le \frac{C }{k^{\min \{ \frac p2-\tau, 1 \}m}}.
		\end{align*}

	\end{proof}
	On the other hand, we can easily get that
	\begin{equation}
		\|R(\omega_k)\|_* \le  \|\omega_k\|_*^{(1+\tau)}.
	\end{equation}
	Using Lemma \ref{lem2.1}, Proposition \ref{pro2.2},  Lemma \ref{lem2.3} and the contraction mapping theorem,  we can obtain the solvability theory for the linearized problem \eqref{linn}.
	
	\begin{proposition} \label{pro2.4}
		Suppose that  $N\ge 3$, and $ V(|y|)$ satisfies $({\bf V_1})$, $(r,h) \in{{\mathscr S}_k}$. There exists an integer  $k_0>0$ large enough, such that for any $k \ge k_0, $ problem \eqref{linn} has a unique solution $\omega_k$ which satisfies
		\begin{equation}\label{11-25-4}
			\|\omega_k\|_*\le C \|l_k\|_* \le  \frac{C }{k^{\min \{ \frac p2-\tau, 1 \}m}}.
		\end{equation}
		and
		\begin{align}\label{estimateforc}
			|c_{k\ell}|\leq  \frac{C} {1+\delta_{\ell 2} r} \le  \frac{C }{k^{\min \{ \frac p2-\tau, 1 \}m}},   \quad{for} ~ \ell=1,2.
		\end{align}
		
		Moreover,  $\omega_k$ can be recognized as a $C^1$ map from $ \mathscr {S}_k$ to $\mathbb{E}_k$, which we denote as $\omega_{r,h}$.
	\end{proposition}
	
	Next, we give an expansion of the energy functional and its partial derivative by the standard arguments. For conciseness, we omit the proof which is quite similar but more accurate than that in \cite{DM}.
	
	\begin{proposition}\label{pro2.5}
		Let $ \omega_{r,h}$ be a function obtained in Proposition \ref{pro2.4}, $(r,h) \in{{\mathscr S}_k}$, and
		\begin{align*}
			F(r,h):= I(W_{r, h}+\omega_{r,h}).
		\end{align*}
		Suppose that $N\ge 3 $ and $V(|y|)$ satisfies $({\bf V_1})$,  then we have the following expansion as $k \to \infty$:
		\begin{equation}
			\begin{aligned}
				F(r,h) &\,=\, k \Big(\frac{A_1}{r^m} + A_2  - 2B_1e^{- 2 \pi \sqrt{1-h^2 } \frac{r}{k} }\left(  2 \pi \sqrt{1-h^2 } \frac{r}{k} \right)^{-\frac{N-1}{2}}  -B_1 e^{-2rh} \left( 2rh \right)^{-\frac{N-1}{2}}  \Big)
				\\
				& \quad+ k\left( O_k\bigg(\frac{1}{r^{m+\gamma}}\bigg)+o_k\left(e^{- 2 \pi \sqrt{1-h^2 } \frac{r}{k} }\left(   2 \pi \sqrt{1-h^2 } \frac{r}{k} \right)^{-\frac{N-1}{2}}\right) \right)
				\\
				& \,=\,k \Big(\frac{A_1}{r^m} + A_2 -2 (B_1+o(1))U(|{x^{+}_{2}}-{x^{+}_{1}}|)-(B_1+o(1))U(|{x^{+}_{1}}-{x^{-}_{1}}|)+ O_k(\frac{1}{r^{m+\gamma}})\Big),
			\end{aligned}
		\end{equation}

		\begin{equation}
			\begin{aligned}\label{Fr}
				\frac{\partial F(r,h)}{\partial r} &\,=\, k \Big(-\frac{m A_1}{r^{m+1}} + \frac{4 B_1\pi}{k}\sqrt{1-h^2 } e^{- 2 \pi \sqrt{1-h^2 } \frac{r}{k} }\left(   2 \pi \sqrt{1-h^2 } \frac{r}{k} \right)^{-\frac{N-1}{2}}  \Big)
				\\
				&\quad+ k \left( O_k\left(\frac{1}{r^{m+1+\gamma}}\right)+o_k\left(\frac{e^{-2 \pi \sqrt{1-h^2 } \frac{r}{k} }}{k} \left(   2 \pi \sqrt{1-h^2 } \frac{r}{k} \right)^{-\frac{N-1}{2}} \right) \right),
			\end{aligned}
		\end{equation}
		and
		\begin{equation}
			\begin{aligned} \label{Fh}
				\frac{\partial F(r,h)}{\partial h}
				&\,=\,   k \bigg(- \frac{4 B_1\pi rh}{k\sqrt{1-h^2}} e^{- 2 \pi \sqrt{1-h^2 } \frac{r}{k} }\left(   2 \pi \sqrt{1-h^2 } \frac{r}{k} \right)^{-\frac{N-1}{2}}  \Big) \\
				&\quad+ 2B_1 r e^{-2rh} \left( 2rh \right)^{-\frac{N-1}{2}} +o_k\left(\frac{\ln k}{k} e^{-2 \pi \sqrt{1-h^2 } \frac{r}{k}} \left(   2 \pi \sqrt{1-h^2 } \frac{r}{k} \right)^{-\frac{N-1}{2}} \right)  \Bigg)
			\end{aligned}
		\end{equation}
		where  $A_1, A_2, B_1$ are  positive constants that
		\begin{equation}\label{AAB}
			A_1 = a \int_{  \mathbb{R}^N}  U^2,   \quad     A_2 =   \big(1- \frac{2}{p+1} \big) \int_{   \mathbb{R}^N}     U^{p+1},  \quad     B_1=    C_0 \int_{  \mathbb{R}^N}  U^p e^{-y_1}.
		\end{equation}
	\end{proposition}
	Since the main term of $F(r,h)$ has a maximum point
	\begin{equation*}
		(\tilde{r_k},\tilde{h_k})=\left( \left(\frac{m}{2\pi}+o(1)\right)k\ln k, \left(\frac{\pi(m+2)}{m}+o(1)\right) \frac{1}{k}\right),
	\end{equation*} which is in the interior of ${{\mathscr S}_k}$, and for $\hat{\sigma}>0$ small, we have
	
	\begin{align}
		&\frac{\partial F(r,h)}{\partial r}\bigg(\left(\frac{m}{2\pi}-\hat{\sigma}\right)k\ln k, h\bigg)>0,\quad \frac{\partial F(r,h)}{\partial r}\bigg(\left(\frac{m}{2\pi}+\hat{\sigma}\right)k\ln k, h\bigg)<0, \quad \nonumber\\ &\text{for any}\ h \in\left(\left(\frac{\pi(m+2)}{m} -\hat{\sigma}\right)\frac{1}{k},\left(\frac{\pi(m+2)}{m} +\hat{\sigma}\right)\frac{1}{k}\right),  \nonumber\\
		&\frac{\partial F(r,h)}{\partial h}\bigg(r, \left(\frac{\pi(m+2)}{m} -\hat{\sigma}\right)\frac{1}{k}\bigg)>0,\quad \frac{\partial F(r,h)}{\partial h}\bigg(r, \left(\frac{\pi(m+2)}{m} +\hat{\sigma}\right)\frac{1}{k}\bigg)<0,\nonumber\\ &\text{for any}\ r \in\left(\left(\frac{m}{2\pi}-\hat{\sigma}\right)k\ln k,\left(\frac{m}{2\pi}-\hat{\sigma}\right)k\ln k\right).
	\end{align}
	Thus for the function $F(r,h)$, we can find a maximum point $(r_k,h_k)$ in the interior of ${{\mathscr S}_k}$.
	
	Therefore,
	we can prove that \eqref{equation} has a solution $u_k$
	of the form
	\[
	u_k = W_{r_k,h_k}+\omega_k,
	\]
	where  $\omega_k\in  \mathbb{E}_k \cap  H^1(\R^N)$,   $(r_k,h_k) \in {{\mathscr S}_k}$ and as $k\to +\infty$,
	
	\begin{equation}\label{1-28-2}
		\|\omega_k\|_* \le  \frac{C }{k^{\min \{ \frac p2-\tau, 1 \}m}}.
	\end{equation}
	
	\medskip
	
	\section{Non-degeneracy of the solutions}
	In the following of the paper, we always assume that $V(|y|)$ satisfies $({\bf V_1})$ and $({\bf V_2})$. Let $u_k$ be a solution of
	\begin{equation}\label{1-27-11}
		-\Delta u + V(|y|)u=u^p\quad \text{in} \ \R^N,
	\end{equation}
	constructed in Theorem~A.
	In the following, we will use local Pohozaev identities to prove Theorem \ref{th11}, arguing by contradiction.

	Suppose that there are $k_n\to +\infty$,  satisfying $\xi_n\in H_s$,
	and
	\begin{equation}\label{10-13-12}
		L_{k_n}\xi_n =0 \quad \text{in} \ \R^N,
	\end{equation}
	while $\xi_n \ne 0$. Without loss of generality, we may assume $\|\xi_n\|_*=1$. Define
	\begin{equation}\label{1-10-4}
		\tilde \xi_{n} (y)= \xi_n
		(y+ x_{k_n, 1}).
	\end{equation}

	\begin{lemma}\label{lem3.1}
		
		It holds
		
		\begin{equation}\label{10-9-4}
			\tilde \xi_{n}  \to  b_1\frac{\partial U}{\partial y_1}+b_3\frac{\partial U}{\partial y_3},
		\end{equation}
		uniformly in $C^1(B_R(0))$ for any $R>0$, where $b_1,b_3$ are  some constants.
		
	\end{lemma}
	
	\begin{proof}
		
		In view of
		$|\tilde \xi_{n}|\le C$, we may assume that $\tilde \xi_{n}\to \xi$ in
		$ C_{loc}(\mathbb R^N)$.   Then $\xi$ satisfies
		
		\begin{equation}\label{3-8-4}
			-\Delta \xi+ \xi  = p U^{p-1}\xi, \quad \text{in}\; \mathbb R^N.
		\end{equation}
		This gives
		
		\begin{equation}\label{4-8-4}
			\xi =\sum_{i=1}^{N} b_{i}\frac{\partial U}{\partial y_i}.
		\end{equation}
		Since $\xi$ is even in $y_j$ for $j=2,4,5,\cdots,N$, it holds that $b_2=b_4=b_5\cdots=b_N=0$.
		
	\end{proof}
	
	According to  \eqref{10-9-4}, we shall decompose
	\begin{equation}\label{xi*}
		\xi_{n} (y)= b_{1n}  \sum_{j=1}^k \bigg(\frac{\partial U_{x^{+}_{j}}}{\partial r}+\frac{\partial U_{x^{-}_{j}}}{\partial r}\bigg)+ \frac{b_{3n}}{r}  \sum_{j=1}^k \bigg(\frac{\partial U_{x^{+}_{j}}}{\partial h}+\frac{\partial U_{x^{-}_{j}}}{\partial h}\bigg)+\xi_n^*,
	\end{equation}
	where
	\[
	\xi_n^*\in \mathbb{E}_k=\Big\{\phi:  \phi\in H_s \cap {C(\R^N)}, \quad \big<U_{x^{+}_{j}}^{p-1}  \overline{\mathbb{Z}}_{\ell j} ,\phi \big> = \big<U_{x^{+}_{j}}^{p-1}  \underline{\mathbb{Z}}_{\ell j} ,\phi \big> = 0,  \ j= 1, \cdots, k, \ \ell= 1,2\Big\}.
	\]
	It follows from Lemma~\ref{lem3.1} that $b_{1n}$ and $b_{3n}$ are bounded.

	\begin{lemma}\label{lem3.2}
		
		It holds
		\[
		\|\xi_n^*\|_*\le \frac{C }{k^{\min \{ \frac p2-\tau, 1 \}m}}.
		\]

	\end{lemma}

	\begin{proof}
		
		We have
		
		\[
		\begin{split}
			L_{k_n} \xi^*_n = & -b_{1n}  \bigl( V(|y|) -1\bigr) \sum_{j=1}^k \bigg(\frac{\partial U_{x^{+}_{j}}}{\partial r}+\frac{\partial U_{x^{-}_{j}}}{\partial r}\bigg)-b_{3n} r^{-1}  \bigl(  V(|y|) -1\bigr) \sum_{j=1}^k \bigg(\frac{\partial U_{x^{+}_{j}}}{\partial h}+\frac{\partial U_{x^{-}_{j}}}{\partial h}\bigg)\\
			&+
			b_{1n}  p  \sum_{j=1}^k \bigl(  u_{k_n}^{p-1}
			-U_{x^{+}_{j}}^{p-1}\bigr)  \frac{\partial U_{x^{+}_{j}}}{\partial r}+b_{3n} r^{-1}  p  \sum_{j=1}^k \bigl(  u_{k_n}^{p-1}
			-U_{x^{+}_{j}}^{p-1}\bigr)  \frac{\partial U_{x^{+}_{j}}}{\partial h}\\
			&+b_{1n}  p  \sum_{j=1}^k \bigl(  u_{k_n}^{p-1}
			-U_{x^{-}_{j}}^{p-1}\bigr)  \frac{\partial U_{x^{-}_{j}}}{\partial r}+b_{3n} r^{-1}  p  \sum_{j=1}^k \bigl(  u_{k_n}^{p-1}
			-U_{x^{-}_{j}}^{p-1}\bigr)  \frac{\partial U_{x^{-}_{j}}}{\partial h}.
		\end{split}
		\]
		
		Similar to the proof of Lemma~\ref{lem2.3}, using \eqref{1-28-2},
		we can  prove
		
		\[
		\bigg\|\bigl(  V(|y|) -1\bigr) \sum_{j=1}^k \bigg(\frac{\partial U_{x^{+}_{j}}}{\partial r}+\frac{\partial U_{x^{-}_{j}}}{\partial r}\bigg) \bigg\|_*\le \frac{C}{ r^m},
		\]
		and
		\[
		\bigg\|\bigl(  V(|y|) -1\bigr)r^{-1} \sum_{j=1}^k \bigg(\frac{\partial U_{x^{+}_{j}}}{\partial h}+\frac{\partial U_{x^{-}_{j}}}{\partial h}\bigg) \bigg\|_*\le \frac{C}{ r^m}.
		\]
		Meanwhile,
		\[
		\bigg\| \sum_{j=1}^k \bigl(  u_{k_n}^{p-1}
		-U_{x^{+}_{j}}^{p-1}\bigr)  \frac{\partial U_{x^{+}_{j}}}{\partial r}+\sum_{j=1}^k \bigl(  u_{k_n}^{p-1}
		-U_{x^{-}_{j}}^{p-1}\bigr)  \frac{\partial U_{x^{-}_{j}}}{\partial r}
		\bigg\|_*\le \frac{C}{ r^m} +Ce^{- \min\{\frac p2-\tau, 1\}(|x_2^+-x_1^+|+|x_1^+-x_1^-|)},
		\]
		and
		\[
		\bigg\| r^{-1}\sum_{j=1}^k \bigl(  u_{k_n}^{p-1}
		-U_{x^{+}_{j}}^{p-1}\bigr)  \frac{\partial U_{x^{+}_{j}}}{\partial h}+r^{-1}\sum_{j=1}^k \bigl(  u_{k_n}^{p-1}
		-U_{x^{-}_{j}}^{p-1}\bigr)  \frac{\partial U_{x^{-}_{j}}}{\partial h}
		\bigg\|_*\le \frac{C}{ r^m} +Ce^{- \min\{\frac p2-\tau, 1\}(|x_2^+-x_1^+|+|x_1^+-x_1^-|)}.
		\]

		Moreover, from  $\xi_n^*\in \mathbb E_k$, similar to Lemma~\ref{lem2.1}, we can prove that there exists $\rho>0$, independent of $k_n$, such that
		
		\[
		\|  L_{k_n} \xi^*_n\|_*\ge  \rho \|  \xi^*_n\|_*.
		\]
		Thus,
		\[
		\|\xi_n^*\|_*\le \frac{C}{ |x^{+}_{1}|^m}+  C e^{- \min\{\frac p2-\tau, 1\}|x^{+}_{2}-x^{+}_{1}|}.
		\]
		
	\end{proof}

	\begin{lemma}\label{lem3.3}
		
		It holds
		
		\begin{equation}\label{1-1-4-12}
			\tilde \xi_{n}\to 0,
		\end{equation}
		uniformly in $C^1(B_R(0))$ for any $R>0$.
		
	\end{lemma}

	\begin{proof}
		
		We apply the identities in Lemma~\ref{l0-14-12}  in the domain  $\Omega=B_{\frac12 |x^{+}_{2}-x^{+}_{1}|}(x^{+}_{1})$.  For simplicity, we drop the subscript $n$.
		
		For $j=1,3$, we define the following bilinear form

		\begin{equation}\label{1-20-3}
			\begin{split}
				I_j(u, \xi, \Omega)= & - \int_{\partial \Omega}\frac{\partial u}{\partial \nu} \frac{\partial \xi}{\partial y_j}- \int_{\partial \Omega}\frac{\partial \xi}{\partial \nu} \frac{\partial u}{\partial y_j}+
				\int_{\partial \Omega} \bigl\langle \nabla u, \nabla \xi\bigr\rangle \nu_j+\int_{\partial \Omega} u \xi \nu_j,
			\end{split}
		\end{equation}
		where $\nu$ is the outward unit normal of $\partial\Omega$.
		Then it follows Lemma~\ref{l0-14-12} that
		
		\begin{equation}\label{2-20-3}
			I_j(u, \xi,  \Omega) = \int_\Omega u\xi \frac{\partial V }{\partial y_j}-\int_{\partial \Omega} \bigl( V(|y|)-1\bigr)u \xi \nu_j+\int_{\partial \Omega} u^p \xi \nu_j.
		\end{equation}
		Combining \eqref{1-27-11} and \eqref{10-13-12}, we can obtain
		\begin{equation}\label{3-20-3}
			\begin{split}
				I_j(u, \xi,  \Omega) = &\int_\Omega \bigl( -\Delta u +   u - u^p\bigr) \frac{\partial \xi}{\partial y_j} +
				\int_\Omega \bigl( -\Delta \xi  +  \xi -p u^{p-1} \xi \bigr)\frac{\partial u }{\partial y_j}+\int_{\partial \Omega} u^p \xi \nu_j.
			\end{split}
		\end{equation}
		
		Noting that
		
		\[
		u_{k_n}=\sum_{j=1}^k  \bigg(U_{x^{+}_{j}}(y)+  U_{x^{-}_{j}}(y)   \bigg)+\omega_{k_n},
		\]
		\begin{equation*}
			\xi_{n} (y)= b_{1n}\sum_{j=1}^k  \bigg(\frac{\partial U_{x^{+}_{j}}(y)}{\partial r}+   \frac{\partial U_{x^{-}_{j}}(y)}{\partial r} \bigg)+  b_{3n}r^{-1}\sum_{j=1}^k  \bigg(\frac{\partial U_{x^{+}_{j}}(y)}{\partial h}+   \frac{\partial U_{x^{-}_{j}}(y)}{\partial h} \bigg)+\xi_n^*,
		\end{equation*}

		and
		\[
		-\Delta \frac{\partial U_{x^{\varsigma}_{j}}}{\partial r}  +  \frac{\partial U_{x^{\varsigma}_{j}}}{\partial r} -p U_{x^{\varsigma}_{j}}^{p-1}\frac{\partial U_{x^{\varsigma}_{j}}}{\partial r}=0,
		\]
		
		\[
		-\Delta \frac{\partial U_{x^{\varsigma}_{j}}}{\partial h}  +  \frac{\partial U_{x^{\varsigma}_{j}}}{\partial h} -p U_{x^{\varsigma}_{j}}^{p-1}\frac{\partial U_{x^{\varsigma}_{j}}}{\partial h}=0.
		\]
		By direct computation, we can  check  that for $j\ge 2$,  $i\ne j$ and $\varsigma \in \{+,-\}$,
		
		\begin{equation}\label{1-27-4}
			\int_\Omega U_{x^{\varsigma}_{j}}^{p-1}\frac{\partial U_{x^{\varsigma}_{j}}}{\partial r}\frac{\partial U_{x^{\varsigma}_{i}} }{\partial y_1}= O\Bigl( e^{-(\frac {p+1}2 -\sigma) |x^{+}_{2}-x^{+}_{1}|}+e^{-(\frac {p+1}2 -\sigma) |x^{+}_{1}-x^{-}_{1}|}\Bigr),
		\end{equation}
		and
		\begin{equation}\label{2-27-4}
			\int_\Omega U_{x^{\varsigma}_{j}}^{p-1}\frac{\partial U_{x^{\varsigma}_{i}}}{\partial r}\frac{\partial U_{x^{\varsigma}_{j}} }{\partial y_1}= O\Bigl( e^{-(\frac {p+1}2 -\sigma) |x^{+}_{2}-x^{+}_{1}|}+e^{-(\frac {p+1}2 -\sigma) |x^{+}_{1}-x^{-}_{1}|}\Bigr),
		\end{equation}
		where $\sigma>0$ is any fixed small constant.
		Since
		\begin{equation*}
			\frac{\partial U_{x^{+}_{1}}}{\partial r}=-\sqrt{1-h^2}\frac{\partial U_{x^{+}_{1}}}{\partial y_1}-h\frac{\partial U_{x^{+}_{1}}}{\partial y_3},
		\end{equation*}
		\begin{equation*}
			\frac{\partial U_{x^{-}_{1}}}{\partial r}=-\sqrt{1-h^2}\frac{\partial U_{x^{-}_{1}}}{\partial y_1}+h\frac{\partial U_{x^{-}_{1}}}{\partial y_3},
		\end{equation*}
		then combining \eqref{3-20-3}--\eqref{2-27-4} and Lemma \ref{lemB.2}, we obtain
		
		\begin{equation}\label{10-27-4}
			\begin{split}
				& I_1\left(\sum_{j=1}^k \bigg(U_{x^{+}_{j}}(y)+  U_{x^{-}_{j}}(y)   \bigg) ,  \sum_{i=1}^k \bigg(\frac{\partial U_{x^{+}_{i}}(y)}{\partial r}+   \frac{\partial U_{x^{-}_{i}}(y)}{\partial r} \bigg),  \Omega\right)
				\\
				=&p \int_\Omega U_{x^{+}_{1}}^{p-1}\frac{\partial U_{x^{+}_{1}}}{\partial r}\Bigl( \frac{\partial U_{x^{+}_{2}} }{\partial y_1} +\frac{\partial U_{x^{+}_{k}} }{\partial y_1}+\frac{\partial U_{x^{-}_{1}} }{\partial y_1}\Bigr)
				-p \int_\Omega U_{x^{+}_{1}}^{p-1}\Bigl(\frac{\partial U_{x^{+}_{2}}}{\partial r}+  \frac{\partial U_{x^{+}_{k}}}{\partial r} +\frac{\partial U_{x^{-}_{1}} }{\partial r}\Bigr)
				\frac{\partial U_{x^{+}_{1}} }{\partial y_1}\\
				&+O\bigg(e^{-(\frac {p+1}2 -\sigma) |x^{+}_{2}-x^{+}_{1}|}+e^{-(\frac {p+1}2 -\sigma) |x^{+}_{1}-x^{-}_{1}|}\bigg)\\
				=&  \int_{\mathbb R^N} U_{x^{+}_{1}}^{p}\frac{\partial }{\partial y_1}\left(\sqrt{1-h^2} \Bigg(\frac{\partial U_{x^{+}_{2}} }{\partial y_1}+\frac{\partial U_{x^{+}_{k}} }{\partial y_1}\bigg)+\bigg(\frac{\partial U_{x^{+}_{2}}}{\partial r}+\frac{\partial U_{x^{+}_{k}}}{\partial r}\bigg)\right)\\
				& +\int_{\mathbb R^N} U_{x^{+}_{1}}^{p}\frac{\partial }{\partial y_3}\left(h\Bigg(\frac{\partial U_{x^{+}_{2}} }{\partial y_1}+\frac{\partial U_{x^{+}_{k}} }{\partial y_1}+\frac{\partial U_{x^{-}_{1}} }{\partial y_1}\bigg)\right) +O\bigg(e^{-(\frac {p+1}2 -\sigma) |x^{+}_{2}-x^{+}_{1}|}+e^{-(\frac {p+1}2 -\sigma) |x^{+}_{1}-x^{-}_{1}|}\bigg)\\
				= &\bigl( D_{11}+o(1)\bigr)\frac{\partial }{\partial y_1} \left(\sqrt{1-h^2} \Bigg(\frac{\partial U_{x^{+}_{2}} }{\partial y_1}+\frac{\partial U_{x^{+}_{k}} }{\partial y_1}\bigg)+\bigg(\frac{\partial U_{x^{+}_{2}}}{\partial r}+\frac{\partial U_{x^{+}_{k}}}{\partial r}\bigg)\right)\Bigg|_{ y= {x^{+}_{1}}}\\
				&+\bigl( D_{12}+o(1)\bigr) \frac{\partial }{\partial y_3}\left(h\Bigg(\frac{\partial U_{x^{+}_{2}} }{\partial y_1}+\frac{\partial U_{x^{+}_{k}} }{\partial y_1}+\frac{\partial U_{x^{-}_{1}} }{\partial y_1}\bigg)\right)\Bigg|_{ y= {x^{+}_{1}}}
				+O\bigl( e^{-(\frac {p+1}2 -\sigma)|x^{+}_{2}-x^{+}_{1}|}\bigr),
			\end{split}
		\end{equation}
		where $D_{11},D_{12}$ are positive constants.
		
		For $\varsigma \in \{+,-\}$, we have
		
		\begin{equation}\label{3-27-4}
			\begin{split}
				&\frac{\partial U_{x^{\varsigma}_{j}}}{\partial r}= \frac{\partial U\bigl(y- ( r\sqrt{1-h^2}\cos\frac{2(j-1)\pi}k, r\sqrt{1-h^2}\sin\frac{2(j-1)\pi}k,\varsigma rh, 0)\bigr)}{\partial r}\\
				=&  -U'(|y- x^{\varsigma}_{j}|) \bigg\langle \frac{y-x^{\varsigma}_{j}}{|y-x^{\varsigma}_{j}|}, \big( \sqrt{1-h^2}\cos\frac{2(j-1)\pi}k, \sqrt{1-h^2}\sin\frac{2(j-1)\pi}k,\varsigma h, 0\big)\bigg\rangle,
			\end{split}
		\end{equation}
		then
		\[
		\begin{split}
			&\quad \sqrt{1-h^2}\frac{\partial U_{x^{+}_{2}} }{\partial y_1}+\frac{\partial U_{x^{+}_{2}}}{\partial r}\\
&=  U'(|y- {x^{+}_{2}}|) \bigg\langle \frac{y-{x^{+}_{2}}}{|y-{x^{+}_{2}}|},\sqrt{1-h^2} \big((1,0,0,0)-( \cos\frac{2\pi}k, \sin\frac{2\pi}k, \frac{h}{\sqrt{1-h^2}},0)\big)\bigg\rangle.
		\end{split}
		\]
		Therefore,

		\[
		\begin{split}
			& \frac{\partial }{\partial y_1}\Bigl( \sqrt{1-h^2}\frac{\partial U_{x^{+}_{2}} }{\partial y_1}+\frac{\partial U_{x^{+}_{2}}}{\partial r}\Bigr)
			\\
			= &  U''(|y- x^{+}_{2}|)\frac{(y-x^{+}_{2})_1}{|y-x^{+}_{2}|} \bigg\langle \frac{y-x^{+}_{2}}{|y-x^{+}_{2}|}, \sqrt{1-h^2} \big((1,0,0,0)-( \cos\frac{2\pi}k, \sin\frac{2\pi}k, \frac{h}{\sqrt{1-h^2}},0)\big)\bigg\rangle\\
			&+U'(|y- x^{+}_{2}|) \bigg\langle \frac{(1,0,0,0)}{|y-x^{+}_{2}|}- (y-x^{+}_{2})\frac{(y-x^{+}_{2})_1}{|y-x^{+}_{2}|^3},\\
 &\sqrt{1-h^2}\big((1,0,0,0)-( \cos\frac{2\pi}k, \sin\frac{2\pi}k, \frac{h}{\sqrt{1-h^2}},0)\big)\bigg\rangle.
		\end{split}
		\]
		Noting that
		\[
		(1, 0,0,0)-( \cos\frac{2\pi}k, \sin\frac{2\pi}k, \frac{h}{\sqrt{1-h^2}},0)= \bigl( \frac{2\pi^2}{k^2}+O\bigl(\frac1{k^4}\bigr), -\frac{2\pi}k+O\bigl(\frac1{k^2}\bigr),-\frac{h}{\sqrt{1-h^2}}, 0\bigr),
		\]
		\[
		|x^{+}_{2}-x^{+}_{1}|=2r_k\sqrt{1-h^2}\sin\frac{\pi}k= \frac {2\pi \sqrt{1-h^2}r_k}k+O\bigl( \frac{ r_k}{k^3}\bigr),
		\]
		\[
		(x^{+}_{2}-x^{+}_{1})_1= r_k\sqrt{1-h^2} \bigl(1- \cos\frac{2\pi}k\bigr)= \sqrt{1-h^2}\frac{2\pi^2 r_k}{k^2}+O\bigl( \frac{ r_k}{k^4}\bigr),
		\]
		\[
		(x^{+}_{2}-x^{+}_{1})_2= -r_k\sqrt{1-h^2} \sin \frac{2\pi}k =- \frac {2\pi \sqrt{1-h^2}r_k}{k}+O\bigl( \frac{r_k}{k^3}\bigr),
		\]
		and
		\[
		(x^{+}_{2}-x^{+}_{1})_3= 0.
		\]
		So we obtain
		
		\[
		\begin{split}
			& \frac{\partial }{\partial y_1}\Bigl(\sqrt{1-h^2} \frac{\partial U_{x^{+}_{2}} }{\partial y_1}+\frac{\partial U_{x^{+}_{2}}}{\partial r}\Bigr)\Bigr|_{y=x^{+}_{1}}
			\\
			= &  U''(|x^{+}_{2}-x^{+}_{1}|)\frac{(x^{+}_{2}-x^{+}_{1})_1}{|x^{+}_{2}-x^{+}_{1}|}  \frac{(x^{+}_{2}-x^{+}_{1})_2}{|x^{+}_{2}-x^{+}_{1}|}\bigl(-\sqrt{1-h^2} \sin\frac{2\pi}k\bigr)\\
			& +U'(|x^{+}_{2}-x^{+}_{1}|) \Bigl( \frac{1}{|x^{+}_{2}-x^{+}_{1}|}  \sqrt{1-h^2}\bigl( 1-\cos\frac{2\pi}k\bigr)  + \frac{(x^{+}_{2}-x^{+}_{1})_2(x^{+}_{2}-x^{+}_{1})_1}{|x^{+}_{2}-x^{+}_{1}|^3}\sqrt{1-h^2} \sin\frac{2\pi}k\Bigr)\\
&+O\bigl(\frac1{k^3}\bigr)\\
			=&  U(|x^{+}_{2}-x^{+}_{1}|)\sqrt{1-h^2}\frac{(x^{+}_{2}-x^{+}_{1})_1}{|x^{+}_{2}-x^{+}_{1}|}\Bigl(\sin\frac{2\pi}k- \frac1{r_k }+ \frac{\sqrt{1-h^2}}{|x^{+}_{2}-x^{+}_{1}|} \sin\frac{2\pi}k
			\Bigr)\Bigl( 1+o(1)\Bigr)\\
			=& \frac{ 2\pi^2 \sqrt{1-h^2}}{k^2}U(|x^{+}_{2}-x^{+}_{1}|)\Bigl( 1+o(1)\Bigr).
		\end{split}
		\]
		And,
		\[
		\begin{split}
			& \frac{\partial }{\partial y_1}\Bigl(\sqrt{1-h^2} \frac{\partial U_{x^{+}_{k}} }{\partial y_1}+\frac{\partial U_{x^{+}_{k}}}{\partial r}\Bigr)\Bigr|_{y=x^{+}_{1}}= \frac{ 2\pi^2 \sqrt{1-h^2}}{k^2}U(|x^{+}_{k}-x^{+}_{1}|)\Bigl( 1+o(1)\Bigr).
		\end{split}
		\]
		
		On the other hand,
		\begin{equation}
			\begin{split}
				\frac{\partial }{\partial y_3}&\left(h\Bigg(\frac{\partial U_{x^{+}_{2}} }{\partial y_1}+\frac{\partial U_{x^{+}_{k}} }{\partial y_1}+\frac{\partial U_{x^{-}_{1}} }{\partial y_1}\bigg)\right)\Bigg|_{y=x^{+}_{1}}\\
				=&h \left( U''(|y- {x^{+}_{2}}|)  \frac{(y- {x^{+}_{2}})_1(y- {x^{+}_{2}})_3}{|y- {x^{+}_{2}}|^2}-U'(|y- {x^{+}_{2}}|)  \frac{(y- {x^{+}_{2}})_1(y- {x^{+}_{2}})_3}{|y- {x^{+}_{2}}|^3}\right)\Bigg|_{y=x^{+}_{1}}\\
				&+h \left( U''(|y- {x^{+}_{k}}|)  \frac{(y- {x^{+}_{k}})_1(y- {x^{+}_{k}})_3}{|y- {x^{+}_{k}}|^2}-U'(|y- {x^{+}_{k}}|)  \frac{(y- {x^{+}_{k}})_1(y- {x^{+}_{k}})_3}{|y- {x^{+}_{k}}|^3}\right)\Bigg|_{y=x^{+}_{1}}\\
				&+h \left( U''(|y- {x^{-}_{1}}|)  \frac{(y-{x^{-}_{1}})_1(y-{x^{-}_{1}})_3}{|y-{x^{-}_{1}}|^2}-U'(|y- {x^{-}_{1}}|)  \frac{(y-{x^{-}_{1}})_1(y-{x^{-}_{1}})_3}{|y-{x^{-}_{1}}|^3}\right)\Bigg|_{y=x^{+}_{1}}\\
				=&0.
			\end{split}
		\end{equation}
		So we have proved
		\begin{equation}\label{60-27-4}
			\begin{split}
				I_1&\left(\sum_{j=1}^k \bigg(U_{x^{+}_{j}}(y)+  U_{x^{-}_{j}}(y)   \bigg) ,  \sum_{i=1}^k \bigg(\frac{\partial U_{x^{+}_{i}}(y)}{\partial r}+   \frac{\partial U_{x^{-}_{i}}(y)}{\partial r} \bigg),  \Omega\right)
				\\
				& = \bigl( D_1+o(1)\bigr)  \frac{ \sqrt{1-h^2}}{k^2}U(|x^{+}_{2}-x^{+}_{1}|)
				+O\bigl( e^{-(\frac {p+1}2 -\sigma)|x^{+}_{2}-x^{+}_{1}|}\bigr),
			\end{split}
		\end{equation}
		where $D_1=2\pi^2D_{11}$.

		Similar to \eqref{1-27-4} and \eqref{2-27-4}, we can  check  that for $j\ge 2$,  $i\ne j$ and $\varsigma \in \{+,-\}$,
		
		\begin{equation}\label{h4}
			\int_\Omega U_{x^{\varsigma}_{j}}^{p-1}\frac{\partial U_{x^{\varsigma}_{j}}}{\partial h}\frac{\partial U_{x^{\varsigma}_{i}} }{\partial y_1}= O\Bigl(r_k e^{-(\frac {p+1}2 -\sigma) |x^{+}_{2}-x^{+}_{1}|}+r_k e^{-(\frac {p+1}2 -\sigma) |x^{+}_{1}-x^{-}_{1}|}\Bigr),
		\end{equation}
		and
		\begin{equation}\label{h5}
			\int_\Omega U_{x^{\varsigma}_{j}}^{p-1}\frac{\partial U_{x^{\varsigma}_{i}}}{\partial h}\frac{\partial U_{x^{\varsigma}_{j}} }{\partial y_1}= O\Bigl(r_k e^{-(\frac {p+1}2 -\sigma) |x^{+}_{2}-x^{+}_{1}|}+r_k e^{-(\frac {p+1}2 -\sigma) |x^{+}_{1}-x^{-}_{1}|}\Bigr).
		\end{equation}
		And,
		\begin{equation*}
			\frac{\partial U_{x^{+}_{1}}}{\partial h}=\frac{rh}{\sqrt{1-h^2}}\frac{\partial U_{x^{+}_{1}}}{\partial y_1}-r\frac{\partial U_{x^{+}_{1}}}{\partial y_3},
		\end{equation*}
		\begin{equation*}
			\frac{\partial U_{x^{-}_{1}}}{\partial h}=\frac{rh}{\sqrt{1-h^2}}\frac{\partial U_{x^{-}_{1}}}{\partial y_1}+r\frac{\partial U_{x^{-}_{1}}}{\partial y_3}.
		\end{equation*}

		Combining \eqref{3-20-3}, \eqref{h4}--\eqref{h5} and Lemma \ref{lemB.2}, we obtain
		
		\begin{equation}\label{10-27-41}
			\begin{split}
				& I_1\left(\sum_{j=1}^k \bigg(U_{x^{+}_{j}}(y)+  U_{x^{-}_{j}}(y)   \bigg) ,  \sum_{i=1}^k \bigg(\frac{\partial U_{x^{+}_{i}}(y)}{\partial h}+   \frac{\partial U_{x^{-}_{i}}(y)}{\partial h} \bigg),  \Omega\right)
				\\
				=&p \int_\Omega U_{x^{+}_{1}}^{p-1}\frac{\partial U_{x^{+}_{1}}}{\partial h}\Bigl( \frac{\partial U_{x^{+}_{2}} }{\partial y_1} +\frac{\partial U_{x^{+}_{k}} }{\partial y_1}+\frac{\partial U_{x^{-}_{1}} }{\partial y_1}\Bigr)
				-p \int_\Omega U_{x^{+}_{1}}^{p-1}\Bigl(\frac{\partial U_{x^{+}_{2}}}{\partial h}+  \frac{\partial U_{x^{+}_{k}}}{\partial h} +\frac{\partial U_{x^{-}_{1}} }{\partial h}\Bigr)
				\frac{\partial U_{x^{+}_{1}} }{\partial y_1}\\
				&+O\bigg(r_k e^{-(\frac {p+1}2 -\sigma) |x^{+}_{2}-x^{+}_{1}|}+ r_k e^{-(\frac {p+1}2 -\sigma) |x^{+}_{1}-x^{-}_{1}|}\bigg)\\
				=&  \int_{\mathbb R^N} U_{x^{+}_{1}}^{p}\frac{\partial }{\partial y_1}\left(-\frac{rh}{\sqrt{1-h^2}} \Bigg(\frac{\partial U_{x^{+}_{2}} }{\partial y_1}+\frac{\partial U_{x^{+}_{k}} }{\partial y_1}+\frac{\partial U_{x^{-}_{1}}}{\partial y_1}\bigg)+\bigg(\frac{\partial U_{x^{+}_{2}}}{\partial h}+\frac{\partial U_{x^{+}_{k}}}{\partial h}+\frac{\partial U_{x^{-}_{1}}}{\partial h}\bigg)\right)\\&+\int_{\mathbb R^N} U_{x^{+}_{1}}^{p}\frac{\partial }{\partial y_3}\left(r\Bigg(\frac{\partial U_{x^{+}_{2}} }{\partial y_1}+\frac{\partial U_{x^{+}_{k}} }{\partial y_1}+\frac{\partial U_{x^{-}_{1}}}{\partial y_1}\bigg)\right)
				+O\bigg(r_k e^{-(\frac {p+1}2 -\sigma) |x^{+}_{2}-x^{+}_{1}|}+r_k e^{-(\frac {p+1}2 -\sigma) |x^{+}_{1}-x^{-}_{1}|}\bigg)\\
				= &\bigl( D_{21}+o(1)\bigr)\frac{\partial }{\partial y_1} \left(-\frac{rh}{\sqrt{1-h^2}} \Bigg(\frac{\partial U_{x^{+}_{2}} }{\partial y_1}+\frac{\partial U_{x^{+}_{k}} }{\partial y_1}+\frac{\partial U_{x^{-}_{1}}}{\partial y_1}\bigg)+\bigg(\frac{\partial U_{x^{+}_{2}}}{\partial h}+\frac{\partial U_{x^{+}_{k}}}{\partial h}+\frac{\partial U_{x^{-}_{1}}}{\partial h}\bigg)\right)\Bigg|_{ y= {x^{+}_{1}}}\\
				&+\bigl( D_{22}+o(1)\bigr) \frac{\partial }{\partial y_3}\left(r \Bigg(\frac{\partial U_{x^{+}_{2}} }{\partial y_1}+\frac{\partial U_{x^{+}_{k}} }{\partial y_1}+\frac{\partial U_{x^{-}_{1}} }{\partial y_1}\bigg)\right)\Bigg|_{ y= {x^{+}_{1}}}
				+O\bigl(r_k e^{-(\frac {p+1}2 -\sigma)|x^{+}_{2}-x^{+}_{1}|}\bigr),
			\end{split}
		\end{equation}
		where $D_{21},D_{22}$ are positive constants.
		
		For $\varsigma \in \{+,-\}$, we have
		\begin{equation}\label{3-27-5}
			\begin{split}
				&\frac{\partial U_{x^{\varsigma}_{j}}}{\partial h}= \frac{\partial U\bigl(y- ( r\sqrt{1-h^2}\cos\frac{2(j-1)\pi}k, r\sqrt{1-h^2}\sin\frac{2(j-1)\pi}k,\varsigma h, 0)\bigr)}{\partial h}\\
				=&  -U'(|y- x^{\varsigma}_{j}|) \bigg\langle \frac{y-x^{\varsigma}_{j}}{|y-x^{\varsigma}_{j}|}, \big( \frac{-rh}{\sqrt{1-h^2}}\cos\frac{2(j-1)\pi}k, \frac{-rh}{\sqrt{1-h^2}}\sin\frac{2(j-1)\pi}k,\varsigma r, 0\big)\bigg\rangle,
			\end{split}
		\end{equation}
		then
		
		\[
		\begin{split}
			&\quad-\frac{rh}{\sqrt{1-h^2}}\frac{\partial U_{x^{+}_{2}} }{\partial y_1}+\frac{\partial U_{x^{+}_{2}}}{\partial h}\\
			&=  U'(|y- {x^{+}_{2}}|) \times\bigg\langle \frac{y-{x^{+}_{2}}}{|y-{x^{+}_{2}}|},-\frac{rh}{\sqrt{1-h^2}}\big((1,0,0,0)-( \cos\frac{2\pi}k, \sin\frac{2\pi}k, -\frac{\sqrt{1-h^2}}{h},0)\big)\bigg\rangle.
		\end{split}
		\]
		Therefore,

		\[
		\begin{split}
			& \frac{\partial }{\partial y_1}\Bigl( -\frac{rh}{\sqrt{1-h^2}}\frac{\partial U_{x^{+}_{2}} }{\partial y_1}+\frac{\partial U_{x^{+}_{2}}}{\partial h}\Bigr)
			\\
			= &  U''(|y- x^{+}_{2}|)\frac{(y-x^{+}_{2})_1}{|y-x^{+}_{2}|} \bigg\langle \frac{y-x^{+}_{2}}{|y-x^{+}_{2}|}, -\frac{rh}{\sqrt{1-h^2}}\big((1,0,0,0)-( \cos\frac{2\pi}k, \sin\frac{2\pi}k, -\frac{\sqrt{1-h^2}}{h},0)\big)\bigg\rangle\\
			&+U'(|y- x^{+}_{2}|) \\&\times\bigg\langle \frac{(1,0,0,0)}{|y-x^{+}_{2}|}- (y-x^{+}_{2})\frac{(y-x^{+}_{2})_1}{|y-x^{+}_{2}|^3}, -\frac{rh}{\sqrt{1-h^2}}\big((1,0,0,0)-( \cos\frac{2\pi}k, \sin\frac{2\pi}k, -\frac{\sqrt{1-h^2}}{h},0)\big)\bigg\rangle.
		\end{split}
		\]
		So we get that
		
		\[
		\begin{split}
			& \frac{\partial }{\partial y_1}\Bigl(-\frac{rh}{\sqrt{1-h^2}}\frac{\partial U_{x^{+}_{2}} }{\partial y_1}+\frac{\partial U_{x^{+}_{2}}}{\partial h}\Bigr)\Bigr|_{y=x^{+}_{1}}
			\\
			= &  U''(|x^{+}_{2}-x^{+}_{1}|)\frac{(x^{+}_{2}-x^{+}_{1})_1}{|x^{+}_{2}-x^{+}_{1}|}  \frac{(x^{+}_{2}-x^{+}_{1})_2}{|x^{+}_{2}-x^{+}_{1}|}\bigl(\frac{rh}{\sqrt{1-h^2}} \sin\frac{2\pi}k\bigr)\\
			&  -\frac{rh}{\sqrt{1-h^2}}U'(|x^{+}_{2}-x^{+}_{1}|) \Bigl( \frac{1}{|x^{+}_{2}-x^{+}_{1}|}  \bigl( 1-\cos\frac{2\pi}k\bigr)  + \frac{(x^{+}_{2}-x^{+}_{1})_2(x^{+}_{2}-x^{+}_{1})_1}{|x^{+}_{2}-x^{+}_{1}|^3} \sin\frac{2\pi}k\Bigr)+O\bigl(\frac1{k^3}\bigr)\\
			=& -\frac{ 2\pi^2 rh}{k^2 \sqrt{1-h^2}}U(|x^{+}_{2}-x^{+}_{1}|)\Bigl( 1+o(1)\Bigr).
		\end{split}
		\]
		Similarly,
		\[
		\begin{split}
			& \frac{\partial }{\partial y_1}\Bigl(-\frac{rh}{\sqrt{1-h^2}}\frac{\partial U_{x^{+}_{k}} }{\partial y_1}+\frac{\partial U_{x^{+}_{k}}}{\partial h}\Bigr)\Bigr|_{y=x^{+}_{1}}= -\frac{ 2\pi^2 rh}{k^2 \sqrt{1-h^2}}U(|x^{+}_{k}-x^{+}_{1}|)\Bigl( 1+o(1)\Bigr),
		\end{split}
		\]
		and
		\[
		\begin{split}
			\frac{\partial }{\partial y_1}\Bigl(-\frac{rh}{\sqrt{1-h^2}}\frac{\partial U_{x^{-}_{1}} }{\partial y_1}+\frac{\partial U_{x^{-}_{1}}}{\partial h}\Bigr)\Bigr|_{y=x^{+}_{1}}= 0.
		\end{split}
		\]

		On the other hand,
		\begin{equation*}
			\frac{\partial }{\partial y_3}\left(r\Bigg(\frac{\partial U_{x^{+}_{2}} }{\partial y_1}+\frac{\partial U_{x^{+}_{k}} }{\partial y_1}+\frac{\partial U_{x^{-}_{1}} }{\partial y_1}\bigg)\right)\Bigg|_{y=x^{+}_{1}} =0.
		\end{equation*}
		So we have proved
		\begin{equation}\label{60-27-67}
			\begin{split}
				I_1&\left(\sum_{j=1}^k \bigg(U_{x^{+}_{j}}(y)+  U_{x^{-}_{j}}(y)   \bigg) ,  \sum_{i=1}^k \bigg(\frac{\partial U_{x^{+}_{i}}(y)}{\partial h}+   \frac{\partial U_{x^{-}_{i}}(y)}{\partial h} \bigg),  \Omega\right)
				\\
				& = -\bigl( D_2+o(1)\bigr)  \frac{rh}{k^2\sqrt{1-h^2}}U(|x^{+}_{2}-x^{+}_{1}|)
				+O\bigl(r_k e^{-(\frac {p+1}2 -\sigma)|x^{+}_{2}-x^{+}_{1}|}\bigr).
			\end{split}
		\end{equation}
		where $D_2=2\pi^2D_{21}$.
		
		Using \eqref{1-28-2}  and Lemma~\ref{lem3.2}, together with $L^p$-estimates for the elliptic equations,
		we can prove

		\[
		|\nabla \omega_{k_n}(y)|,\;\; |\nabla \xi_n^*(y)|\le \frac{C }{k^{\min \{ \frac p2-\tau, 1 \}m}},\quad \forall y\in\partial\Omega,
		\]
		from which, we see
		
		\begin{equation}\label{70-27-4}
			I_1\left(\omega_{k_n},  \sum_{i=1}^k \bigg(\frac{\partial U_{x^{+}_{j}}(y)}{\partial r}+   \frac{\partial U_{x^{-}_{j}}(y)}{\partial r} \bigg),  \Omega\right)= O\bigg( \frac{1 }{k^{\min \{ \frac p2-\tau, 1 \}m}} \bigg)e^{-\frac{1-\sigma}2
				|x^{+}_{2}-x^{+}_{1}|},
		\end{equation}
		
		\begin{equation}\label{70-27-41}
			I_1\left(\omega_{k_n},  \sum_{i=1}^k \bigg(\frac{\partial U_{x^{+}_{j}}(y)}{\partial h}+   \frac{\partial U_{x^{-}_{j}}(y)}{\partial h} \bigg),  \Omega\right)= O\bigg( \frac{1 }{k^{\min \{ \frac p2-\tau, 1 \}m}} \bigg)r_ke^{-\frac{1-\sigma}2
				|x^{+}_{2}-x^{+}_{1}|},
		\end{equation}
		and

		\begin{equation}\label{71-27-4}
			I_1\left(\sum_{j=1}^k \bigg(U_{x^{+}_{j}}(y)+  U_{x^{-}_{j}}(y)\bigg)  ,  \xi^*_n,  \Omega\right)= O\bigg( \frac{1 }{k^{\min \{ \frac p2-\tau, 1 \}m}} \bigg)e^{-\frac{1-\sigma}2
				|x^{+}_{2}-x^{+}_{1}|}.
		\end{equation}
		Thus,
		\begin{equation}\label{I1}
			\begin{split}
				I_1( u_{k_n} ,  \xi_n,  \Omega)
				=&b_{1n}\bigl( D_1+o(1)\bigr)  \frac{ \sqrt{1-h^2}}{k^2}U(|x^{+}_{2}-x^{+}_{1}|)-b_{3n}\bigl( D_2+o(1)\bigr)  \frac{h}{k^2\sqrt{1-h^2}}U(|x^{+}_{2}-x^{+}_{1}|)\\
				&+O\bigl( e^{-(\frac {p+1}2 -\sigma)|x^{+}_{2}-x^{+}_{1}|}\bigr)+ O\bigg( \frac{1 }{k^{\min \{ \frac p2-\tau, 1 \}m}} \bigg)e^{-\frac{1-\sigma}2|x^{+}_{2}-x^{+}_{1}|}.
			\end{split}
		\end{equation}

		It is easy to check that

		\begin{equation}\label{1-28-4}
			\int_{\partial \Omega} \bigl( V(|y|)-1\bigr)u \xi \nu_1=O\bigg(\frac1{r_k^m} e^{-(1-\sigma)|x^{+}_{2}-x^{+}_{1}|}\bigg),
		\end{equation}
		and
		
		\begin{equation}\label{2-28-4}
			\int_{\partial \Omega} u^p \xi \nu_1=O\bigl(e^{-(\frac{p+1}2-\sigma)|x^{+}_{2}-x^{+}_{1}|}\bigr).
		\end{equation}
		Combining \eqref{2-20-3}, and \eqref{I1}--\eqref{2-28-4}, we obtain
		
		\begin{equation}\label{3-28-4}
			\begin{split}
				&b_{1n}\bigl( D_1+o(1)\bigr)  \frac{ \sqrt{1-h^2}}{k^2}U(|x^{+}_{2}-x^{+}_{1}|)-b_{3n}\bigl( D_2+o(1)\bigr)  \frac{h}{k^2\sqrt{1-h^2}}U(|x^{+}_{2}-x^{+}_{1}|)\\
				&=\int_\Omega u\xi \frac{\partial V }{\partial y_1}+O\bigg( \frac{1 }{k^{\min \{ \frac p2-\tau, 1 \}m}} e^{-\frac{1-\sigma}2|x^{+}_{2}-x^{+}_{1}|}+e^{-(\frac {p+1}2 -\sigma)|x^{+}_{2}-x^{+}_{1}|}\bigg)\\
				&=\int_\Omega u\xi \frac{\partial V }{\partial y_1}+O\bigg( \frac{1 }{k^{({\frac{1-\sigma}2+\min \{ \frac p2-\tau, 1 \})m}}} \bigg).
			\end{split}
		\end{equation}

		\medskip
		
		For $I_3( u_{k_n} ,  \xi_n,  \Omega)$, we can similarly calculate that
		\begin{equation}\label{80-27-4}
			\begin{split}
				I_3&\left(\sum_{j=1}^k \bigg(U_{x^{+}_{j}}(y)+  U_{x^{-}_{j}}(y)   \bigg) ,  \sum_{i=1}^k \bigg(\frac{\partial U_{x^{+}_{i}}(y)}{\partial r}+   \frac{\partial U_{x^{-}_{i}}(y)}{\partial r} \bigg),  \Omega\right)
				\\
				& = \bigl( D_3+o(1)\bigr)h   U(|x^{+}_{1}-x^{-}_{1}|)
				+O\bigl( e^{-(\frac {p+1}2 -\sigma)|x^{+}_{1}-x^{-}_{1}|}\bigr),
			\end{split}
		\end{equation}
		
		and
		\begin{equation}\label{90-27-67}
			\begin{split}
				I_3&\left(\sum_{j=1}^k \bigg(U_{x^{+}_{j}}(y)+  U_{x^{-}_{j}}(y)   \bigg) ,  \sum_{i=1}^k \bigg(\frac{\partial U_{x^{+}_{i}}(y)}{\partial h}+   \frac{\partial U_{x^{-}_{i}}(y)}{\partial h} \bigg),  \Omega\right)
				\\
				& = \bigl( D_4+o(1)\bigr)  r U(|x^{+}_{1}-x^{-}_{1}|)
				+O\bigl( e^{-(\frac {p+1}2 -\sigma)|x^{+}_{1}-x^{-}_{1}|}\bigr).
			\end{split}
		\end{equation}
		where $D_3,D_4$ are positive constants.
		Then
		\begin{equation}\label{I3}
			\begin{split}
				I_3(u_{k_n} ,  \xi_n,  \Omega)
				=&b_{1n}\bigl( D_3+o(1)\bigr)  hU(|x^{+}_{1}-x^{-}_{1}|)+b_{3n}\bigl( D_4+o(1)\bigr)   U(|x^{+}_{1}-x^{-}_{1}|)\\
				&+O\bigl( e^{-(\frac {p+1}2 -\sigma)|x^{+}_{1}-x^{-}_{1}|}\bigr)+ O\bigg( \frac{1 }{k^{\min \{ \frac p2-\tau, 1 \}m}} \bigg)e^{-\frac{1-\sigma}2|x^{+}_{1}-x^{-}_{1}|}.
			\end{split}
		\end{equation}
		Thus,
		\begin{equation}\label{3-28-5}
			\begin{split}
				&b_{1n}\bigl( D_3+o(1)\bigr)  hU(|x^{+}_{1}-x^{-}_{1}|)+b_{3n}\bigl( D_4+o(1)\bigr)  U(|x^{+}_{1}-x^{-}_{1}|)\\
				&=\int_\Omega u\xi \frac{\partial V }{\partial y_3}+O\bigg( \frac{1 }{k^{\min \{ \frac p2-\tau, 1 \}m}} e^{-\frac{1-\sigma}2|x^{+}_{1}-x^{-}_{1}|}+e^{-(\frac {p+1}2 -\sigma)|x^{+}_{1}-x^{-}_{1}|}\bigg)\\
				&=\int_\Omega u\xi \frac{\partial V }{\partial y_3}+O\bigg( \frac{1 }{k^{({\frac{1-\sigma}2+\min \{ \frac p2-\tau, 1 \})m}}} \bigg).
			\end{split}
		\end{equation}

		Using \eqref{1-28-2}, \eqref{xi*}  and  Lemma~\ref{lem3.2}, we find

		\begin{equation}\label{1000-21-12}
			\begin{split}
				&\int_{ \Omega} u\xi V'(|y|) |y|^{-1} y_1
				\\
				=&\bigg(-\sqrt{1-h^2}b_{1n}+\frac{h}{\sqrt{1-h^2}}b_{3n}\bigg)\int_{ \mathbb R^N} U_{x^{+}_{1}}   \frac{\partial U_{x^{+}_{1}}}{\partial y_1} V'(|y|) |y|^{-1} y_1 \\ &-\bigg(h b_{1n}+b_{3n}\bigg)\int_{ \mathbb R^N} U_{x^{+}_{1}}   \frac{\partial U_{x^{+}_{1}}}{\partial y_3} V'(|y|) |y|^{-1} y_1+\frac1{r_k^{m+1}} O\Bigl(\frac{1}{ r_k^m}+   e^{- \min\{\frac p2-\sigma, 1\}|x^{+}_{2}-x^{+}_{1}|}\Bigr)\\
				=&\bigg(-\sqrt{1-h^2}b_{1n}+\frac{h}{\sqrt{1-h^2}}b_{3n}\bigg)\int_{ \mathbb R^N} U_{x^{+}_{1}}   \frac{\partial U_{x^{+}_{1}}}{\partial y_1} V'(|y|) |y|^{-1} y_1\\
 &+\frac1{r_k^{m+1}} O\Bigl(\frac{1}{ r_k^m}+   e^{- \min\{\frac p2-\sigma, 1\}|x^{+}_{2}-x^{+}_{1}|}\Bigr).
			\end{split}
		\end{equation}

		For any $m>1$, it holds
		
		\[
		\frac{ y_1+ x_{1, 1}^{+}} {|y+x^{+}_{1}|^\alpha} = \frac{  x_{1, 1}^{+}} {|x^{+}_{1}|^\alpha}- (\alpha(1-h^2)-1) \frac{ y_1} {|x^{+}_{1}|^\alpha}+O\bigl(\frac1{|x^{+}_{1}|^{\alpha+1}}\bigr).
		\]
		Thus,
		
		\[
		\int_{\Omega}U  \frac{\partial U}{\partial y_1}  \frac{ y_1+ x_{1, 1}^{+}} {|y+x^{+}_{1}|^\alpha}=- (\alpha(1-h^2)-1)\int_{\Omega}U  \frac{\partial U}{\partial y_1} \frac{ y_1} {|x^{+}_{1}|^\alpha}
		+O\bigl(\frac1{|x^{+}_{1}|^{\alpha+1}}\bigr).
		\]
		Therefore, 
		
		\begin{equation}\label{11-21-12}
			\begin{split}
				&\int_{ \mathbb R^N} U_{x^{+}_{1}}   \frac{\partial U_{x^{+}_{1}}}{\partial y_1} V'(|y|) |y|^{-1} y_1\\
				= & \int_{ \mathbb R^N+ x^{+}_{1} } U  \frac{\partial U}{\partial y_1}V'(|y+ x^{+}_{1}|) |y+ x^{+}_{1}|^{-1} (y_1+  x_{1,1})\\
				= & \int_{ \mathbb R^N } U  \frac{\partial U}{\partial y_1} \Bigl( -\frac{ m a_1 }{ |y+ x^{+}_{1}|^{m+1}} -\frac{ (m+1) a_2 } {|y+ x^{+}_{1}|^{m+2}} +
				O\bigl( \frac{ 1 }{ |y+ x^{+}_{1}|^{m+3}}\bigr)\Bigr)\frac{ y_1+  x_{1,1}}{|y+ x^{+}_{1}| }
				\\
				=&\frac {m\big((m+2)(1-h^2)-1\big)a_1}{r_{k}^{m+2}} \Bigl( \int_{\mathbb R^N}  U U'(|y|) |y|^{-1} y_1^2  +o(1)\Bigr).
			\end{split}
		\end{equation}

		Combining \eqref{3-28-4} and \eqref{11-21-12}, we obtain

		\begin{equation}\label{36}
			\begin{split}
				& \frac {m\big((m+2)(1-h^2)-1\big)a_1}{r_k^{m+2}}\bigg(-\sqrt{1-h^2}b_{1n}+\frac{h}{\sqrt{1-h^2}}b_{3n}\bigg) \Bigl( \int_{\mathbb R^N}  U U'(|y|) |y|^{-1} y_1^2   +o(1)\Bigr)
				\\
				=& b_{1n}\bigl( D_1+o(1)\bigr)  \frac{ \sqrt{1-h^2}}{k^2}U(|x^{+}_{2}-x^{+}_{1}|)-b_{3n}\bigl( D_2+o(1)\bigr)  \frac{h}{k^2\sqrt{1-h^2}}U(|x^{+}_{2}-x^{+}_{1}|)\\
				&+\frac1{r_k^{m+1}} O\Bigl(\frac{1}{ r_k^m}+   e^{- \min\{\frac p2-\sigma, 1\}|x^{+}_{2}-x^{+}_{1}|}\Bigr)
				+O\bigg( \frac{1 }{k^{({\frac{1-\sigma}2+\min \{ \frac p2-\tau, 1 \})m}}} \bigg).
			\end{split}
		\end{equation}
		On the other hand, from \eqref{3-28-5} we have

		\begin{equation}\label{37}
			\begin{split}
				& \frac {m\big(1-(m+2)h^{2}\big)a_1}{r_k^{m+2}}\bigg(h b_{1n}+b_{3n}\bigg) \Bigl( \int_{\mathbb R^N}  U U'(|y|) |y|^{-1} y_3^2   +o(1)\Bigr)
				\\
				=& b_{1n}\bigl( D_3+o(1)\bigr)  hU(|x^{+}_{1}-x^{-}_{1}|)+b_{3n}\bigl( D_4+o(1)\bigr)  U(|x^{+}_{1}-x^{-}_{1}|)\\
				&+\frac1{r_k^{m+1}} O\Bigl(\frac{1}{ r_k^m}+   e^{- \min\{\frac p2-\sigma, 1\}|x^{+}_{1}-x^{-}_{1}|}\Bigr)
				+O\bigg( \frac{1 }{k^{({\frac{1-\sigma}2+\min \{ \frac p2-\tau, 1 \})m}}} \bigg).
			\end{split}
		\end{equation}

		Since $m>\max\bigg\{\displaystyle\frac 4{p-1}, 4\bigg\}$, then
		
		\[
		\frac1{r_k^{m+1}} O\Bigl(\frac{1}{ r_k^m}+   e^{- \min\{\frac p2-\sigma, 1\}|x^{+}_{2}-x^{+}_{1}|}\Bigr)
		+O\bigg( \frac{1 }{k^{({\frac{1-\sigma}2+\min \{ \frac p2-\tau, 1 \})m}}} \bigg)=o\Bigl(  \frac{ 1}{r_k^{m+2}}  \Bigr).
		\]
		By \eqref{B1}--\eqref{B13},
		
		\begin{equation}\label{10-28-4}
			\frac{ k}{r_k^{m+1}}
			\sim U(|x^{+}_{2}-x^{+}_{1}|),\qquad\frac{ 1}{k r_k^{m+1}}
			\sim U(|x^{+}_{1}-x^{-}_{1}|).
		\end{equation}
		Hence,
		\begin{equation}\label{40}
			\frac{  U(|x^{+}_{2}-x^{+}_{1}|)}{k^2}\sim U(|x^{+}_{1}-x^{-}_{1}|)\sim \frac1k \frac{1 }{r_k^{m+1}}=\frac{r_k}k \frac{1 }{r_k^{m+2}}>>\frac{1 }{r_k^{m+2}}.
		\end{equation}
		Thus, from \eqref{36}--\eqref{40}, we get
		\begin{equation*}
			b_{1n}=b_{3n}= o(1).
		\end{equation*}
	\end{proof}
	
	\medskip
	
	\begin{proof} [{\bf Proof of Theorem~\ref{th11}}]
		
		Let  $G(y, x)$ be the Green's function of $-\Delta u + V(|y|) u$  in $\mathbb R^N$. We have

		\begin{equation}\label{2-28-11}
			\xi_k (y)= p \int_{\mathbb R^N}  G(z, y) u_k^{p-1}(z) \xi_k(z)\,dz.
		\end{equation}
		Then
		\[
		\begin{split}
			&\bigg|\int_{\mathbb R^N}  G(z, y) u_k^{p-1}(z) \xi_k(z)\,dz\bigg|\le  C\|\xi _k\|_{*} \int_{\mathbb R^N} G(z, y) u_k^{p-1}(z)\sum_{j=1}^k (e^{-\tau|z-x_{j}^{+}|}+e^{-\tau|z-x_{j}^{-}|})\,dz
			\\
			\le & C\|\xi _k\|_{*} \int_{\mathbb R^N} G(z, y)\bigl( W_{r,h}^{p-1}(z)+ |\omega|^{p-1}\bigr)\sum_{j=1}^k (e^{-\tau|z-x_{j}^{+}|}+e^{-\tau|z-x_{j}^{-}|})\,dz\\
			\le & C\|\xi _k\|_{*} \int_{\mathbb R^N} G(z, y) \Bigl(\sum_{j=1}^k (e^{- (p-1-\sigma+\tau)|z-x_{j}^{+}|}+e^{- (p-1-\sigma+\tau)|z-x_{j}^{-}|})\Bigr)\\&+C\|\xi _k\|_{*}\int_{\mathbb R^N} G(z, y)\|\omega\|_*^{p-1} \Bigl(\sum_{j=1}^k (e^{-(p\tau-\sigma) |z-x_{j}^{+}|}+e^{-(p\tau-\sigma) |z-x_{j}^{-}|})\Bigr)
			\\
			\le & C\|\xi_k\|_{*}\Bigl(\sum_{j=1}^k (e^{- (p-1-\sigma+\tau)|z-x_{j}^{+}|}+e^{- (p-1-\sigma+\tau)|z-x_{j}^{-}|})\Bigr)+C\|\xi_k\|_{*} \sum_{j=1}^k (e^{-(p\tau-\sigma)|y-x_{j}^{+}|}+e^{-(p\tau-\sigma)|y-x_{j}^{-}|}),
		\end{split}
		\]
		where $\sigma\in (0, \tau)$ is any fixed small constant. So we obtain

		\[
		\begin{split}
			&\frac{ |\xi_k(y)|}{ \sum_{j=1}^k (e^{-\tau|y-x_{j}^{+}|}+e^{-\tau|y-x_{j}^{-}|})}\\\le &C\|\xi_k\|_{*}\frac{\sum_{j=1}^k (e^{- (p-1-\sigma+\tau)|z-x_{j}^{+}|}+e^{- (p-1-\sigma+\tau)|z-x_{j}^{-}|})+ \sum_{j=1}^k (e^{-(p\tau-\sigma)|y-x_{j}^{+}|}+e^{-(p\tau-\sigma)|y-x_{j}^{-}|})}{ \sum_{j=1}^k (e^{-\tau|y-x_{j}^{+}|}+e^{-\tau|y-x_{j}^{-}|})}.
		\end{split}
		\]
		
		Since $\xi_k\to 0$ in $B_R(x_{j}^{+})$  and $\|\xi_k\|_*=1$, we know that $\frac{ |\xi_k(y)|}{ \sum_{j=1}^k (e^{-\tau|y-x_{j}^{+}|}+e^{-\tau|y-x_{j}^{-}|})}$
		attains its maximum in $\mathbb R^N \setminus \cup_{j=1}^k B_R(x_{j}^{+})$. Thus,
		
		\[
		\|\xi_k\|_* \le o(1)\|\xi_k\|_* .
		\]
		So $\|\xi_k\|_*\to 0$ as $k\to +\infty$. This is a contradiction to $\|\xi_k\|_*=1$.

	\end{proof}

	\medskip
	
	\section{Existence of new double-tower solutions}
	Let  $u_k$ be the  solutions constructed in Section \ref{sec2}, where $k>0$ is a large even integer. Since $k$ is even, $u_k$ is even in each
	$y_j$, $j=1,\cdots, N$.   Moreover,  $u_k$  is radial in $y''=(y_4, \cdots, y_N)$. Recall that
	\[
	p_{j}^{+}=t\Bigl(0, 0, 0,\sqrt{1-l^2}\cos\frac{2(j-1)\pi}n, \sqrt{1-l^2}\sin\frac{2(j-1)\pi}n,l,0\Bigr),\quad j=1,\cdots,n,
	\]
	\[
	p_{j}^{-}=t\Bigl(0, 0, 0,\sqrt{1-l^2} \cos\frac{2(j-1)\pi}n, \sqrt{1-l^2}\sin\frac{2(j-1)\pi}n,-l,0\Bigr),\quad j=1,\cdots,n,
	\]
	where   $(t,l) \in [t_0n\ln n, t_1n\ln n]\times[\displaystyle\frac{l_0}{n},\displaystyle\frac{l_1}{n}]$.
	And,
	\[
	\begin{split}
		X_s=\Bigg\{ u:  u\in H_s, u\;& \text{is even in} \;y_h, h=1,\cdots,N, \  \ \\
&u\left(y_1, y_2,y_3, t\sqrt{1-l^2}\cos\theta , t\sqrt{1-l^2}\sin\theta,y_6, y^*\right)\\
		& =
		u\left(y_1,y_2,y_3, t\sqrt{1-l^2}\cos\left(\theta+\frac{2\pi j}n\right) , t\sqrt{1-l^2}\sin\left(\theta+\frac{2\pi j}n\right),y_6, y^*\right)
		\Bigg\}.
	\end{split}
	\]
	Here $y^*=(y_7,\cdots,y_N) \in \R^{N-6}$.
	For simplicity, we denote $Y_{t,l}(y)=\sum_{j=1}^n (U_{p_{j}^{+}}+U_{p_{j}^{-}})(y).$
	
	Noting that both $u_k $  and $\sum_{j=1}^n (U_{p_{j}^{+}}+U_{p_{j}^{-}})$ belong to $ X_s$, while  $u_k$  and $\sum_{j=1}^n (U_{p_{j}^{+}}+U_{p_{j}^{-}})$ are separated  from each other. We aim to construct a solution for \eqref{equation} of the form
	\[
	u= u_k +Y_{t,l}(y) +\xi,
	\]
	where $\xi\in X_s$ is a small perturbed term.  Since the procedure to construct such solutions are similar to that in Section \ref{sec2}, we just sketch it.
	
	We consider the linearized problem:
	\begin{align}\label{linnn}
		\begin{cases}
			-\Delta {\xi}+V(|y|)\xi-p{(u_k+Y_{t,l}(y))^{p-1}}\xi =g+\sum\limits_{j=1}^n\sum\limits_{\ell=3}^4 c_{n\ell} \Big( U_{p_{j}^{+}}^{p-1}\overline{\mathbb{Z}}_{\ell j}+ U_{p_{j}^{-}}^{p-1}\underline{\mathbb{Z}}_{\ell j}\Big)
			\;\;
			\text{in}\;
			\R^N,
			\\[2mm]
			\xi\in \mathbb{F}_n,
		\end{cases}
	\end{align}
	for some constants $c_{n\ell} $, where the functions $\overline{\mathbb{Z}}_{\ell j}$ and $\underline{\mathbb{Z}}_{\ell j},\ell=3,4$ are given by
	\begin{align*}
		\overline{\mathbb{Z}}_{3j}=\frac{\partial U_{p_{j}^{+}}}{\partial t},
		\qquad
		\underline{\mathbb{Z}}_{3j}=\frac{\partial U_{p_{j}^{-}}}{\partial t},
		\qquad
		\overline{\mathbb{Z}}_{4j}=\frac{\partial U_{p_{j}^{+}}}{\partial l},
		\qquad
		\underline{\mathbb{Z}}_{4j}=\frac{\partial U_{p_{j}^{-}}}{\partial l}.
	\end{align*}
	for $ j= 1, \cdots, k $. Moreover the function $\xi$ belongs to the set $\mathbb{F}_k$ given by
	\begin{align}\label{SpaceF}
		\mathbb{F}_n= \Big\{\xi:  \xi\in X_s \cap {C(\R^N)}, \quad \big<U_{p_{j}^{+}}^{p-1}  \overline{\mathbb{Z}}_{\ell j} ,\xi \big> = \big<{U_{p_{j}^{-}}^{p-1}}  \underline{\mathbb{Z}}_{\ell j} ,\xi \big> = 0,  \ j= 1, \cdots, k, \ \ell= 3,4\Big\}.
	\end{align}

	\begin{lemma}\label{lem5.1}
		Suppose that $\xi_{n}$ solves \eqref{linnn} for $g=g_{n}$. If $\| g_n\|_{H^1(\R^N)} \to 0$  as $n\to +\infty$, so does  $\|\xi_{n}\|_{{H^1(\R^N)}}$.
	\end{lemma}
	\begin{proof}
		For simplicity, we will use  $\| u\|$ to denote $\|u \|_{H^1(\mathbb R^N)}$. We argue by contradiction.  Suppose that there are $p_{n,j}^{\varsigma}$,   $g_n$  and $\xi_n$, $\varsigma \in \{+,-\}$ satisfying \eqref{linnn},    $\|g_n\|\to 0$ and $\|\xi_n\|\ge c''>0$.
		We may assume  $\|\xi_n\|^2 =n$, then $\|g_n\|^2 = o(n)$.

		First, we estimate  $c_{n \ell}$.  We have
		\begin{equation}
			c_{n \ell} \sum_{j=1}^n \int_{\mathbb R^N}  \bigg( U_{p_{n,j}^{+}}^{p-1}\overline{\mathbb{Z}}_{\ell j}+U_{p_{n,j}^{-}}^{p-1}\overline{\mathbb{Z}}_{\ell j}\bigg) \overline{\mathbb{Z}}_{q 1}=  \bigl\langle -\Delta {\xi_n}+V(|y|)\xi_n-p{(u_k+Y_{t,l}(y))^{p-1}}\xi_n -g_n,   \overline{\mathbb{Z}}_{q 1}\bigr\rangle, \ell, q =3,4.
		\end{equation}
		Noting that
		\begin{equation*}
			\int_{\mathbb R^N}  -\Delta {\xi_n}\overline{\mathbb{Z}}_{q 1}+V(|y|)\xi_n\overline{\mathbb{Z}}_{q 1}-p U_{p_{n,1}^{+}}^{p-1}\xi_n\overline{\mathbb{Z}}_{q 1} =0,
		\end{equation*}
		and for $n$ large enough, we have
		\begin{align*}
			\int_{\mathbb R^N} \bigg( p{(u_k+Y_{t,l}(y))^{p-1}}-p U_{p_{n,1}^{+}}^{p-1} \bigg)\xi_n\overline{\mathbb{Z}}_{q 1}& =O\big(e^{-|p_{n,1}^{+}-p_{n,2}^{+}|+\tau}+e^{-|p_{n,1}^{+}-x_{1}^{+}|+\tau}\big)(1+\delta_{q4} t)\|\xi_n\|\\
			& =O\big(e^{-|p_{n,1}^{+}-p_{n,2}^{+}|+\tau}\big)(1+\delta_{q4} t) \|\xi_n\|.
		\end{align*}
		On the other hand,
		\begin{eqnarray*}
			\big< U_{p_{n,1}^{+}}^{p-1}\overline{\mathbb{Z}}_{\ell 1}  , \overline{\mathbb{Z}}_{q 1} \big> =\left\{\begin{array}{rcl}  o_k(1),\qquad \text{if}\quad \ell \neq q,
				&&\\[2mm]
				\space\\[2mm]
				\bar{c'} (1+\delta_{q4} t^2),\qquad \text{if} \quad\ell=q. \end{array}\right.
		\end{eqnarray*}
		Thus,
		\begin{equation}
			c_{n \ell}=\frac{1+\delta_{q4} t}{1+\delta_{q4} t^2}\left(\big(e^{-|p_{n,1}^{+}-p_{n,2}^{+}|+\tau}\big) \|\xi_n\|+\|\g_n\|\right)=o(\sqrt{n}).
		\end{equation}

		Meanwhile,  from $\xi_n \in \mathbb{F}_n$, it is standard to prove that
		\[ \xi_n( y+ p_{n, j}) \to  0,\quad \text{in}\;  H^1_{loc}(\mathbb R^N).
		\]
		Moreover, since  $\displaystyle\frac{1}{\sqrt n} \xi_n$ is bounded in $H^1(\mathbb R^N)$, we can assume that
		\[
		\frac{1}{\sqrt n} \xi_n\rightharpoonup \xi,\quad \text{weakly in }\; H^1_{loc}(\mathbb R^N),
		\]
		and
		\[
		\frac{1}{\sqrt n} \xi_n\to  \xi,\quad \text{strongly in }\; L^2_{loc}(\mathbb R^N).
		\]
		Thus, $\xi$ satisfies
		\[
		-\Delta \xi+V(|y|) \xi -pu_k^{p-1}\xi =0  \quad \text{in}\; \mathbb R^N.
		\]
		By  Theorem~\ref{th11}, $\xi=0$.  Therefore,
		\[
		\int_{\mathbb R^N}  \bigl(  u_k+Y_{t,l}\bigr)^{p-1}\xi_n^2=o(n).
		\]
		We also have
		\[
		\langle  g_n,   \xi_n\bigr\rangle= o(n),
		\]
		and
		\[
		\Bigl\langle   c_{n\ell} \Big( U_{p_{n,j}^{+}}^{p-1}\overline{\mathbb{Z}}_{\ell j}+ U_{p_{n,j}^{-}}^{p-1}\underline{\mathbb{Z}}_{\ell j}\Big), \xi_n \Bigr\rangle  =o(n).
		\]
		So we obtain
		\[
		\begin{split}
			& \int_{\mathbb R^N} |\nabla \xi_n|^2+V(|y|) | \xi_n|^2 \\
			=&  \bigg\langle p{(u_k+Y_{t,l}(y))^{p-1}}\xi_n +g_n+\sum\limits_{j=1}^n\sum\limits_{\ell=3}^4 c_{n\ell} \Big( U_{p_{n,j}^{+}}^{p-1}\overline{\mathbb{Z}}_{\ell j}+ U_{p_{n,j}^{-}}^{p-1}\underline{\mathbb{Z}}_{\ell j}\Big),  \xi_n\bigg\rangle \\
			=&
			o(n).
		\end{split}
		\]
		This is a contradiction.

	\end{proof}
	
	\medskip
	
	Deriving from Lemma \ref{lem5.1}, we can obtain the existence, uniqueness results of linearized problem \eqref{linnn}:
	\begin{proposition}\label{pro5.2}
		There exist  $n_0>0  $ and a constant  $C>0 $ such
		that for all  $n\ge n_0 $ and all  $g_n\in H^{1}(\R^N) $, problem
		$(\ref{linnn}) $ has a unique solution  $\xi_n\equiv {\bf \Xi}_n(g_n) $.
	\end{proposition}
	
	We can rewrite problem \eqref{linnn} as
	\begin{align}\label{linnnn}
		\begin{cases}
			{\bf L_n} \xi_n =  {\bf l_n}  + {\bf R}(\xi_n)+ \sum\limits_{j=1}^n\sum\limits_{\ell=3}^4 c_{n\ell} \Big( U_{p_{j}^{+}}^{p-1}\overline{\mathbb{Z}}_{\ell j}+ U_{p_{j}^{-}}^{p-1}\underline{\mathbb{Z}}_{\ell j}\Big),\\[2mm]
			\xi_n  \in  \mathbb{F}_k,
		\end{cases}
	\end{align}
	where
	\begin{equation*}
		{\bf L_n} \xi_n :=-\Delta {\xi_n}+V(|y|)\xi_n-p{(u_k+Y_{t,l}(y))^{p-1}}\xi_n,
	\end{equation*}
	
	\begin{equation*}
		{\bf l_n} :=-\sum_{j=1}^k (V(|y|)-1) (U_{p^{+}_{j}}+U_{p^{-}_{j}}) +\bigg( (u_k+Y_{t,l}(y))^{p}-u_k^p-\sum_{j=1}^k \big(U_{p^{+}_{j}}^p+U_{p^{-}_{j}}^p\big)\bigg),
	\end{equation*}
	and
	\begin{equation*}
		{\bf R_n}(\xi_n):=(u_k+Y_{t,l}(y)+\xi_n)^p  -(u_k+Y_{t,l}(y))^{p}  -p{(u_k+Y_{t,l}(y))^{p-1}}\xi_n.
	\end{equation*}
	
	Next, we estimate ${\bf l_n}$ and ${\bf R}(\xi_n)$ in \eqref{linnnn}.
	
	\begin{lemma}\label{lem5.3}
		For all $(t,l) \in \mathscr {T}_n$, there is a small $\sigma>0$ that
		\begin{equation}\label{l_n}
			\| {\bf l_n}\| \le  \frac{C }{n^{\frac{m+1}{2}+\sigma}}.
		\end{equation}

	\end{lemma}
	
	\begin{proof}
		Divide $\R^N$ into $2n$ parts:
		\begin{equation*}
			\widetilde \Omega_j=\Bigl\{ y=(y',y_4, y_5,y_6 ,y^*)\in\R^3\times \mathbb R^3\times \R^{N-6}:
			\Bigl\langle \frac {( y_4, y_5)}{|(y_4, y_5)|}, (\cos \frac{2(j-1)}{n},\sin \frac{2(j-1)}{n})\Bigr\rangle\ge \cos \frac{\pi}{n}\Bigr\},
		\end{equation*}
		and
		\begin{equation*}
			\widetilde \Omega_j^{+}=\Bigl\{ y:y=(y',y_4, y_5,y_6 ,y^*)\in\widetilde \Omega_j,y_6\ge 0 \Bigr\},
		\end{equation*}
		
		\begin{equation*}
			\widetilde \Omega_j^{-}=\Bigl\{ y:y=(y',y_4, y_5,y_6 ,y^*)\in\widetilde \Omega_j,y_6< 0 \Bigr\}.
		\end{equation*}
		
		Recall that
		\begin{equation*}
			\langle  {\bf l_n}, \varphi \rangle= -\int_{\R^N} \sum_{j=1}^k (V(|y|)-1) (U_{p^{+}_{j}}+U_{p^{-}_{j}}) \varphi+\int_{\R^N}\bigg( (u_k+Y_{t,l}(y))^{p}-u_k^p-\sum_{j=1}^k \big(U_{p^{+}_{j}}^p+U_{p^{-}_{j}}^p\big)\bigg)\varphi.
		\end{equation*}
		For $5\le j\le \displaystyle\frac{n}{2},y\in \widetilde \Omega_j^{+}$ , we have
		\begin{equation*}
			|y-p_1^{+}|\ge (4-\tau)\frac{t\pi}{n}.
		\end{equation*}
		Then
		\begin{align*}
			& \int_{\R^N} \sum_{j=1}^k (V(|y|)-1) (U_{p^{+}_{j}}+U_{p^{-}_{j}}) \varphi\\
			=&2n \int_{\R^N} (V(|y|)-1) U_{p^{+}_{1}} \varphi\\
			\le&C n \bigg(\int_{\cup^4_{i=1}\widetilde\Omega_i} \sum_{j=1}^k (V(|y|)-1) (U_{p^{+}_{j}}+U_{p^{-}_{j}}) \varphi+\int_{\cup^\frac{n}{2}_{i=5}\widetilde\Omega_i} \sum_{j=1}^k (V(|y|)-1) (U_{p^{+}_{j}}+U_{p^{-}_{j}}) \varphi\bigg)\\
			\le&C n \left(\frac{1}{t^{m-\frac{1}{2}}} +\frac{1}{t^{4m-2\tau'}}\right)||\varphi||_{L^2(\R^N)}
			\le C \frac{n}{t^{m-\frac{1}{2}}} ||\varphi||_{L^2(\R^N)}\\
			\le&C \frac{1}{n^{\frac{m+1}{2}+\sigma}} ||\varphi||.
		\end{align*}
		after fixing $\tau'>0$ small enough.
		
		On the other hand,
		\begin{align*}
			&\int_{\R^N}\bigg( (u_k+Y_{t,l}(y))^{p}-u_k^p-\sum_{j=1}^n \big(U_{p^{+}_{j}}^p+U_{p^{-}_{j}}^p\big)\bigg)\varphi\\
			\le&C(k)n\int_{\widetilde\Omega_1^{+}}\bigg( (u_k+Y_{t,l}(y))^{p}-u_k^p-Y_{t,l}(y)^{p}+Y_{t,l}(y)^{p}-\sum_{j=1}^n \big(U_{p^{+}_{j}}^p+U_{p^{-}_{j}}^p\big)\bigg)\varphi\\
			\le&C n \int_{\widetilde\Omega_1^{+}}\bigg( (u_k+Y_{t,l}(y))^{p}-u_k^p-Y_{t,l}(y)^{p}\bigg)+Cn\sum_{j=2}^n\int_{\widetilde\Omega_1^{+}} U_{p^{+}_{1}}^{p-1} U_{p^{+}_{j}} |\varphi|\\
			\le&C n \int_{\widetilde\Omega_1^{+}}\bigg( (u_k+Y_{t,l}(y))^{p}-u_k^p-Y_{t,l}(y)^{p}\bigg)+  Cn^{\frac{1}{2}}\sum_{j=2}^n \bigg( \int_{\widetilde\Omega_1^{+}}  U_{p^{+}_{1}}^{2p-2} U_{p^{+}_{j}}^2  \bigg)^{\frac{1}{2}}||\varphi||.
		\end{align*}
		If $p< 2$, then
		\begin{align*}
			\int_{\widetilde\Omega_1^{+}}\bigg( (u_k+Y_{t,l}(y))^{p}-u_k^p-Y_{t,l}(y)^{p}\bigg)
			\le C \int_{\widetilde\Omega_1^{+}} u_k^{\frac{p}{2}} Y_{t,l}(y)^{\frac{p}{2}}
			=O(e^{(-\frac{p}{2}+\sigma)|{p^{+}_{1}-{x^{+}_{1}}|}}),
		\end{align*}
		and
		\begin{align*}
			\sum_{j=2}^n \bigg( \int_{\widetilde\Omega_1^{+}}  U_{p^{+}_{1}}^{2p-2} U_{p^{+}_{j}}^2  \bigg)
			\le\sum_{j=2}^n \bigg( e^{(-\frac{p}{2}+\sigma)|{p^{+}_{j}-{p^{+}_{1}}}|}\bigg)\le\frac{C}{n^{(\frac{p}{2}-\sigma)m-\sigma'}}.
		\end{align*}
		If $p\ge 2$, then
		\begin{align*}
			\int_{\widetilde\Omega_1^{+}}\bigg( (u_k+Y_{t,l}(y))^{p}-u_k^p-Y_{t,l}(y)^{p}\bigg)
			\le C \int_{\widetilde\Omega_1^{+}} \bigg(u_k^{p-1} Y_{t,l}(y)+u_kY_{t,l}(y)^{p-1}\bigg)
			=O\big(e^{(-\frac{p}{2}+\sigma)|{p^{+}_{1}-{x^{+}_{1}}|}}\big),
		\end{align*}
		and for $\sigma'>0$ small enough,
		\begin{align*}
			\sum_{j=2}^n \bigg( \int_{\widetilde\Omega_1^{+}}  U_{p^{+}_{1}}^{2p-2} U_{p^{+}_{j}}^2  \bigg)
			\le\sum_{j=2}^n \bigg( e^{-|{p^{+}_{j}-{p^{+}_{1}}}|}\bigg)\le\frac{C}{n^{m-\sigma'}}.
		\end{align*}
		Thus, for $n>>k$, we have
		\begin{align*}
			&\int_{\R^N}\bigg( (u_k+Y_{t,l}(y))^{p}-u_k^p-\sum_{j=1}^n \big(U_{p^{+}_{j}}^p+U_{p^{-}_{j}}^p\big)\bigg)\varphi\\
			\le& \frac{C}{n^{\min\{\frac{p}{2}-\sigma,1\}m-\frac{1}{2}-\sigma'}}||\varphi||+O(n e^{(-\frac{p}{2}+\sigma)|{p^{+}_{1}-{x^{+}_{1}}|}})||\varphi|| \\
			\le&\frac{C}{n^{\min\{\frac{p}{2}-\sigma,1\}m-\frac{1}{2}-\sigma'}}||\varphi||.
		\end{align*}
		Noting that $m>\max\bigg\{\displaystyle\frac{4}{p-1},4\bigg\}$, then we have
		\begin{equation*}
			\min\bigg\{\frac{p}{2}-\sigma,1\bigg\}m-\frac{1}{2}-\sigma'>\frac{m+1}{2}.
		\end{equation*}
		Therefore,
		\begin{equation*}
			\bigg|\int_{\R^N}\bigg( (u_k+Y_{t,l}(y))^{p}-u_k^p-\sum_{j=1}^n \big(U_{p^{+}_{j}}^p+U_{p^{-}_{j}}^p\big)\bigg)\varphi \bigg|\le  C \frac{1}{n^{\frac{m+1}{2}+\sigma}} ||\varphi||.
		\end{equation*}
	\end{proof}

	\begin{lemma}\label{lem5.4}
		It holds
		\begin{equation}\label{R_n1}
			\|{\bf R_n}(\xi)\|\le C\|\xi\|^{\min\{p, 2\}}.
		\end{equation}
		Moreover,
		\begin{equation}\label{R_n2}
			\|{\bf R_n}(\xi_1)- {\bf R_n}(\xi_2)\|\le C \Bigl(  \|\xi_1\|^{\min\{p-1, 1\}} +\|\xi_2\|^{\min\{p-1, 1\}}\Bigr)\|\xi_1-\xi_2\|.
		\end{equation}

	\end{lemma}
	
	\medskip
	
	Using Proposition \ref{pro5.2} and Lemmas~\ref{lem5.3}-\ref{lem5.4}, we can prove the following proposition in a standard way.

	\begin{proposition}\label{pro5.5}
		There is an integer $n_0>0$, such that for each $n\ge n_0$  and  $(t,l) \in \mathscr {T}_n$,   \eqref{linnnn} has a solution $\xi_n$ for some constants
		$c_{n \ell}$, and
		\[
		\|\xi_n\|\le \frac{C }{n^{\frac{m+1}{2}+\sigma}},
		\]
		where $\sigma>0$ is a small constant.\\
		
		Moreover,  $\xi_n$ can be recognized as a $C^1$ map from $ \mathscr {T}_n$ to $\mathbb{F}_n$, which we denote as $\xi_{t,l}$.
	\end{proposition}
	
	\begin{remark}
		Since $k$ is fixed, as $n \to \infty$, $u_k$ always plays a negligible role in the estimations of ${\bf l_n}$ and ${\bf R_n}(\xi)$ for $n>>k$, and that's why the results of Lemmas \ref{lem5.3}-\ref{lem5.4} and Proposition \ref{pro5.5} are quite similar to that of Proposition 2.3 and Lemma 2.4 in \cite{GMPY1}.
	\end{remark}

	Next, Let $ \xi_{t,l}$ be a function obtained in Proposition \ref{pro5.5}, $(t,l) \in{{\mathscr T}_n}$. We need to find a critical point of the energy functional of the form $u_k +Y_{t,l}(y) +\xi_{t,l}$. Define
	\begin{equation}
		G(t,l)=I(u_k +Y_{t,l}(y) +\xi_{t,l}).
	\end{equation}
	
	By a standard calculation which is similar to the calculation of $F(r,h)$, we have the expansion of $G(t,l)$:
	\begin{proposition}\label{pro5.7}
		Suppose that $N\ge 6, V(|y|)$ satisfies $({\bf V_1})$ and $({\bf V_2})$,  then we have the following expansion as $n \to \infty:$
		
		\begin{equation}
			\begin{aligned}\label{GGG}
				G(t,l)
				=&I(u_k) + n \Bigg(\frac{A_1}{t^m} + A_2 -2 (B_1+o(1))U(|{p^{+}_{2}}-{p^{+}_{1}}|)\\
				&\qquad\qquad -(B_1+o(1))U(|{p^{+}_{1}}-{p^{-}_{1}}|)+ O_n\left(\frac{1}{t^{m+2\sigma}}\right)\Bigg),
			\end{aligned}
		\end{equation}
		\begin{equation}
			\begin{aligned}\label{Gt}
				\frac{\partial G(t,l)}{\partial t} &\,=\, n \Big(-\frac{m A_1}{t^{m+1}} + \frac{4 B_1\pi}{n}\sqrt{1-l^2 } e^{- 2 \pi \sqrt{1-l^2 } \frac{t}{n} }\left(   2 \pi \sqrt{1-l^2 } \frac{t}{n} \right)^{-\frac{N-1}{2}}  \Big) \\
				&\quad+ n \left( O_n\left(\frac{1}{t^{m+1+2\sigma}}\right)+o_n\left(\frac{e^{-2 \pi \sqrt{1-l^2 } \frac{t}{n} }}{n} \left(   2 \pi \sqrt{1-l^2 } \frac{t}{n} \right)^{-\frac{N-1}{2}} \right) \right),
			\end{aligned}
		\end{equation}
		and
		\begin{equation}
			\begin{aligned} \label{Gl}
				\frac{\partial G(t,l)}{\partial l}
				&\,=\,   n \bigg(- \frac{4 B_1\pi tl}{n\sqrt{1-l^2}} e^{- 2 \pi \sqrt{1-l^2 } \frac{t}{n} }\left(   2 \pi \sqrt{1-l^2 } \frac{t}{n} \right)^{-\frac{N-1}{2}}  \Big) \\
				&\quad+ 2B_1 t e^{-2tl} \left( 2tl \right)^{-\frac{N-1}{2}}+ o_n\left(\frac{\ln n}{n} e^{-2 \pi \sqrt{1-l^2 } \frac{t}{n}} \left(  2 \pi \sqrt{1-l^2 } \frac{t}{n} \right)^{-\frac{N-1}{2}} \right)  \Bigg)
			\end{aligned}
		\end{equation}
		where  $A_1, A_2, B_1$ are  positive constants in \eqref{AAB}.
		
	\end{proposition}
	\begin{proof}
		By some calculation, we have
		\begin{equation}\label{G1}
			G(t,l) =I(u_k(y))+I(Y_{t,l})+\frac{1}{2} \langle{\bf{L_n}}\xi_n,\xi_n\rangle-\langle{\bf {l_n}},\xi_n\rangle-\frac{1}{p+1}({\bf \bar R_{n,1}}-{\bf \bar R_{n,2}}),
		\end{equation}
		where
		\begin{equation}\label{G2}
			\begin{aligned}
				{\bf \bar R_{n,1}}=&\displaystyle\int_{\R^N}\bigg((u_k+Y_{t,l}+\xi_n)^{p+1}-(u_k+Y_{t,l})^{p+1}-(p+1)(u_k+Y_{t,l})^p\xi_n-\frac{(p+1)p}{2}(u_k+Y_{t,l})^{p-1}\xi_n^2\bigg)\\
				=&O(||\xi_n||^{p+1}_{H^1(\R^N)})=O(||\xi_n||^{2}_{H^1(\R^N)}),
			\end{aligned}
		\end{equation}
		and
		\begin{equation}\label{G3}
			{\bf \bar R_{n,2}}=\displaystyle\int_{\R^N} \bigg(u_k^{p+1}+Y_{t,l}^{p+1}+(p+1)u_k^pY_{t,l} -  (u_k+Y_{t,l})^{p+1}\bigg)=O(e^{{-(\frac{p+1}{2}+\sigma)}|p_1^{+}-x_1^{+}|})=o_n\left(\frac{1}{t^{m+1+2\sigma}}\right),
		\end{equation}
		for $n>>k$ being large enough.
		Combining \eqref{G1}--\eqref{G3}, we obtain
		\begin{equation}\label{G4}
			\begin{aligned}
				G(t,l) =&I(u_k(y))+I(Y_{t,l})+\frac{1}{2} \langle{\bf{L_n}}\xi_n,\xi_n\rangle-\langle{\bf {l_n}},\xi_n\rangle-\frac{1}{p+1}({\bf \bar R_{n,1}}-{\bf \bar R_{n,2}})\\
				=&I(u_k(y))+I(Y_{t,l})+O(||{\bf {l_n}}|| \ ||\xi_n||+||\xi_n||^{2})+o_n\left(\frac{1}{t^{m+1+2\sigma}}\right)\\
				=&I(u_k(y))+I(Y_{t,l})+O_n\left(\frac{1}{t^{m+1+2\sigma}}\right).
			\end{aligned}
		\end{equation}
		The following calculation is standard, we just need to proceeding exactly as in Proposition \ref{pro2.5}, then we can get \eqref{GGG} and the expression \eqref{Gt} and \eqref{Gl} for its partial derivatives.

	\end{proof}

	\begin{proof} [{\bf Proof of Theorem~\ref{th1.2}}]
		According to Proposition \ref{pro5.7}, we can find  main term of $G(t,l)$ has a maximum point
		\begin{equation*}
			(\tilde{t_n},\tilde{l_n})=\left( \left(\frac{m}{2\pi}+o(1)\right)n\ln n, \left(\frac{\pi(m+2)}{m}+o(1)\right) \frac{1}{n}\right),
		\end{equation*} which is in the interior of ${{\mathscr T}_k}$. Similar to the argument in Section \ref{sec2},
		$G(t,l)$ has a maximum point $(t_n,l_n)$ in the interior of ${{\mathscr T}_k}$.
		
		Therefore,
		we can prove that \eqref{equation} has a solution $u_{k,n}$
		of the form
		\[
		u_{k,n} = u_k+Y_{t_n,l_n}+\xi_n.
		\]

	\end{proof}

	\medskip

	\appendix
	\section{Pohozaev identities}\label{appendixA}
	Let $u>0$ solve

	\begin{equation}\label{1-26-11}
		-\Delta u + V(|y|)u= u^p, \  \hbox{ in } \mathbb{R}^{N},
	\end{equation}
	and let $\xi$ solve
	
	\begin{equation}\label{2-26-11}
		-\Delta \xi +  V(|y|)\xi =p u^{p-1}\xi \  \hbox{ in }\mathbb{R}^{N}.
	\end{equation}
	The following identity can be found in \cite{GMPY1}.

	\begin{lemma}\label{l0-14-12}
		Assume that $\Omega $ is a smooth bounded domain in $\mathbb{R}^{N}$, then We have
		
		\begin{equation}\label{1-20-12}
			\begin{split}
				- \int_{\partial \Omega}\frac{\partial u}{\partial \nu} \frac{\partial \xi}{\partial y_i}- \int_{\partial \Omega}\frac{\partial \xi}{\partial \nu} \frac{\partial u}{\partial y_i}+
				\int_{\partial \Omega} \bigl\langle \nabla u, \nabla \xi\bigr\rangle \nu_i +\int_{\partial \Omega}V(|y|) u \xi \nu_i  -\int_{\partial \Omega} u^p \xi \nu_i=\int_\Omega u\xi \frac{\partial V }{\partial y_i}
			\end{split}
		\end{equation}

	\end{lemma}

	\medskip

	\section{Basic estimates}\label{appendixB}
	\begin{proposition}
		Suppose that $V(|y|)$ satisfies $({\bf{V_1}})$ and $({\bf{V_2}}),(r,h)\in \mathscr{S}_k,m> \displaystyle\max\left\{\frac{4}{p-1},4\right\}$ then
		\begin{equation}\label{B1}
			\frac{a(m+1)\sqrt{1-h^2}(1+o(1))}{2r_k^{m+1}}\int_{\R^N}U(|y|)^2=(4B_1+o(1))U(|x_{2}^{+}-x_{1}^{+}|)\frac{\pi}{k},
		\end{equation}
		and  \begin{equation}\label{B13}
			\frac{a(m+1)h(1+o(1))}{2r_k^{m+1}}\int_{\R^N}U(|y|)^2=(B_1+o(1))U(|x_{1}^{+}-x_{1}^{-}|).
		\end{equation}

		\begin{proof}
			Since
			
			\[
			\int_{\mathbb R^N}  \bigg(  \nabla ( W_{r_k,h_k}+\omega_k )\nabla \frac{\partial U_{x_{1}^{+}}}{\partial y_1} + V(y) ( W_{r_k,h_k}+\omega_k )  \frac{\partial U_{x_{1}^{+}}}{\partial y_1}- ( W_{r_k,h_k}+\omega_k )^p
			\frac{\partial U_{x_{1}^{+}}}{\partial y_1}\bigg)=0,
			\]
			then
			
			\begin{align*}
				& \int_{\mathbb R^N}  ( V(y) -1)( W_{r_k,h_k}+\omega_k )  \frac{\partial U_{x_{1}^{+}}}{\partial y_1}\\=& \int_{\mathbb R^N}\left( ( W_{r_k,h_k}+\omega_k )^p- \sum_{j=1}^k \bigg(U^p_{x_{j}^{+}}+U^p_{x_{j}^{-}}\bigg)\right)
				\frac{\partial U_{x_{1}^{+}}}{\partial y_1}+\int_{\mathbb R^N}\omega_k U_{x_{1}^{+}}^{p-1} \frac{\partial U_{x_{1}^{+}}}{\partial y_1}\\&=\int_{\mathbb R^N}\left( ( W_{r_k,h_k}+\omega_k )^p- \sum_{j=1}^k \bigg(U^p_{x_{j}^{+}}+U^p_{x_{j}^{-}}\bigg)\right)
				\frac{\partial U_{x_{1}^{+}}}{\partial y_1}.
			\end{align*}

			It is easy to check that

			\[
			\begin{split}
				&\int_{\mathbb R^N}  \bigl( ( V(y) -1)( W_{r_k,h_k}+\omega_k )  \frac{\partial U_{x_{1}^{+}}}{\partial y_1}\\
				=&\int_{\mathbb R^N}   ( V(y) -1)U_{x_{1}^{+}}  \frac{\partial U_{x_{1}^{+}}}{\partial y_1}+ O\left( \frac{e^{-(1-\tau)(|x_{2}^{+}-x_{1}^{+}|+|x_{1}^{+}-x_{1}^{-}|)}) }{r_k^{m}}+ || \omega_k||_* \left( \frac{1 }{r_k^{m}}+e^{-(1-\tau)(|x_{2}^{+}-x_{1}^{+}|+|x_{1}^{+}-x_{1}^{-}|)}\right)\right)\\
				=&-\frac12 \int_{\mathbb R^N}   \frac{\partial  V(y) }{\partial y_1}U_{x_{1}^{+}}^2+ O\Bigl( \frac{1 }{r_k^{2m}}+ e^{- \min\{ p-\tau, 2 \}|x_{2}^{+}-x_{1}^{+}|}+e^{- \min\{ p-\tau, 2 \}|x_{1}^{+}-x_{1}^{-}|}\Bigr)\\
				=&\frac{a (m+1) x_{1,1}^{+}}{2r_k^{m+2}}\int_{\mathbb R^N} U(|y|)^2 + O\Bigl(\frac{1 }{r_k^{m+2}}+ e^{- \min\{ p-\tau, 2 \}|x_{2}^{+}-x_{1}^{+}|}\Bigr)\\
				=&\frac{a (m+1) \sqrt{1-h^2}}{2r_k^{m+1}}\int_{\mathbb R^N} U(|y|)^2 + O\Bigl(\frac{1 }{r_k^{m+2}}+ e^{- \min\{ p-\tau, 2 \}|x_{2}^{+}-x_{1}^{+}|}\Bigr).
			\end{split}
			\]
			Moreover,

			\[
			\begin{split}
				&\int_{\mathbb R^N}\left( ( W_{r_k,h_k}+\omega_k )^p- \sum_{j=1}^k \bigg(U^p_{x_{j}^{+}}+U^p_{x_{j}^{-}}\bigg)\right)
				\frac{\partial U_{x_{1}^{+}}}{\partial y_1}\\
				=&p\int_{\mathbb R^N} U_{x_{1}^{+}}^{p-1} \bigg(\sum_{j=2}^k U_{x_{j}^{+}}+\sum_{j=1}^k U_{x_{j}^{-}}\bigg)
				\frac{\partial U_{x_{1}^{+}}}{\partial y_1}+O\Bigl(  \frac{1 }{r_k^{2m}}+  e^{- \min\{ p-\tau, 2 \}|x_{2}^{+}-x_{1}^{+}|}+e^{- \min\{ p-\tau, 2 \}|x_{1}^{+}-x_{1}^{-}|}\Bigr)\\
				=&-\int_{\mathbb R^N} U_{x_{1}^{+}}^{p} \bigg(\sum_{j=2}^k
				\frac{\partial U_{x_{j}^{+}}}{\partial y_1}+\sum_{j=1}^k
				\frac{\partial U_{x_{j}^{-}}}{\partial y_1}\bigg)+O\Bigl( \frac{1 }{r_k^{2m}}+ e^{- \min\{ p-\tau, 2 \}|x_{2}^{+}-x_{1}^{+}|}+e^{- \min\{ p-\tau, 2 \}|x_{1}^{+}-x_{1}^{-}|}\Bigr)\\
				=& \bigl(2B_1 +o(1)\bigr)U(|x_{2}^{+}-x_{1}^{+}|)\Bigl( \frac{(x_{1}^{+}-x_{2}^{+})_1}{|x_{1}^{+}-x_{2}^{+}|}+\frac{(x_{1}^{+}-x_{k}^{+})_1}{|x_{1}^{+}-x_{k}^{+}|}\Bigr)\\
				&+O\Bigl( \frac{1 }{r_k^{2m}}+ e^{- \min\{ p-\tau, 2 \}|x_{2}^{+}-x_{1}^{+}|}+e^{- \min\{ p-\tau, 2 \}|x_{1}^{+}-x_{1}^{-}|}\Bigr).
			\end{split}
			\]
			Therefore, we have

			\begin{equation}\label{1-18-4}
				\begin{split}
					&\frac{a (m+1)\sqrt{1-h^2}}{2r_k^{m+1}}\int_{\mathbb R^N} U(|y|)^2
					\\
					=&  \bigl(2B_1 +o(1)\bigr) U(|x_{2}^{+}-x_{1}^{+}|)\Bigl( \frac{(x_{2}^{+}-x_{1}^{+})_1}{|x_{2}^{+}-x_{1}^{+}|}+\frac{(x_{k}^{+}-x_{1}^{+})_1}{|x_{k}^{+}-x_{1}^{+}|}\Bigr)\\
					&+O\Bigl( \frac{1 }{r_k^{2m}}+\frac1{r_k^{m+2}}+ e^{- \min\{ p-\tau, 2 \}|x_{2}^{+}-x_{1}^{+}|}\Bigr)\\
					=& \bigl(4B_1 +o(1)\bigr) U(|x_{2}^{+}-x_{1}^{+}|)\frac{\pi}{k}+O\Bigl( \frac{1 }{r_k^{2m}}+\frac{1}{r_k^{m+2}}+ e^{- \min\{ p-\tau, 2 \}|x_{2}^{+}-x_{1}^{+}|}\Bigr).
				\end{split}
			\end{equation}
			Noting that   $e^{-|x_{2}^{+}-x_{1}^{+}|}\sim \displaystyle\frac{1 }{k^{m}}$  and $m>\displaystyle\frac4{p-1}$, it holds
			
			\[
			e^{- \min\{ p-\tau, 2 \}|x_{2}^{+}-x_{1}^{+}|}\le   \frac{C }{r_k^{\min\{ p-\tau, 2 \}m}}= o\bigg(\frac{1 }{r_k^{m+1}}\bigg),
			\]
			then we get \eqref{B1}.
			
			On the other hand, we can similarly obtain that
			\begin{equation*}
				\int_{\mathbb R^N}  \bigl( ( V(y) -1)( W_{r_k,h_k}+\omega_k )  \frac{\partial U_{x_{1}^{+}}}{\partial y_3}
				=\frac{a (m+1) h}{2r_k^{m+1}}\int_{\mathbb R^N} U(|y|)^2 + O\Bigl(\frac{1 }{r_k^{m+2}}+ e^{- \min\{ p-\tau, 2 \}|x_{2}^{+}-x_{1}^{+}|}\Bigr),
			\end{equation*}
			and
			\[
			\begin{split}
				&\int_{\mathbb R^N}\left( ( W_{r_k,h_k}+\omega_k )^p- \sum_{j=1}^k \bigg(U^p_{x_{j}^{+}}+U^p_{x_{j}^{-}}\bigg)\right)
				\frac{\partial U_{x_{1}^{+}}}{\partial y_3}\\
				=&-\int_{\mathbb R^N} U_{x_{1}^{+}}^{p} \bigg(\sum_{j=2}^k
				\frac{\partial U_{x_{j}^{+}}}{\partial y_3}+\sum_{j=1}^k
				\frac{\partial U_{x_{j}^{-}}}{\partial y_3}\bigg)+O\Bigl( \frac{1 }{r_k^{2m}}+ e^{- \min\{ p-\tau, 2 \}|x_{2}^{+}-x_{1}^{+}|}+e^{- \min\{ p-\tau, 2 \}|x_{1}^{+}-x_{1}^{-}|}\Bigr)\\
				=& \bigl(B_1 +o(1)\bigr)U(|x_{1}^{+}-x_{1}^{-}|)+O\Bigl( \frac{1 }{r_k^{2m}}+ e^{- \min\{ p-\tau, 2 \}|x_{2}^{+}-x_{1}^{+}|}+e^{- \min\{ p-\tau, 2 \}|x_{1}^{+}-x_{1}^{-}|}\Bigr).
			\end{split}
			\]
			Then from
			\begin{align*}
				\int_{\mathbb R^N}  ( V(y) -1)( W_{r_k,h_k}+\omega_k )  \frac{\partial U_{x_{1}^{+}}}{\partial y_3}=\int_{\mathbb R^N}\left( ( W_{r_k,h_k}+\omega_k )^p- \sum_{j=1}^k \bigg(U^p_{x_{j}^{+}}+U^p_{x_{j}^{-}}\bigg)\right)
				\frac{\partial U_{x_{1}^{+}}}{\partial y_3},
			\end{align*}
			we have
			\begin{equation}
				\frac{a (m+1) h}{2r_k^{m+1}}\int_{\mathbb R^N} U(|y|)^2 =\bigl(B_1 +o(1)\bigr)U(|x_{1}^{+}-x_{1}^{-}|)+O\Bigl( \frac{1 }{r_k^{2m}}+\frac{1 }{r_k^{m+2}}+ e^{- \min\{ p-\tau, 2 \}|x_{2}^{+}-x_{1}^{+}|}\Bigr).
			\end{equation}
			Since $h\sim \displaystyle\frac{1}{k}$ and $m>\displaystyle\frac{4}{p-1},$ then
			\begin{equation}
				\frac{1 }{r_k^{m+2}}=o\bigg(\frac{h }{r_k^{m+1}}\bigg),\qquad e^{- \min\{ p-\tau, 2 \}|x_{2}^{+}-x_{1}^{+}|}= o\bigg(\frac{h }{r_k^{m+1}}\bigg).
			\end{equation}
			Thus, \eqref{B13} holds.

		\end{proof}
	\end{proposition}
	
	\medskip

	\begin{lemma}\label{lemB.2}
		Suppose that $(r,h)\in \mathscr{S}_k$, $\Omega=B_{\frac{1}{2}|x^+_2-x^+_1|(x^+_1)}$, we have for $j=1,\cdots,k$.
		\begin{equation}\label{4-20-3}
			\begin{split}
				I_\ell\bigg(U_{x^{+}_{j}} ,  \frac{\partial U_{x^{+}_{j}}}{\partial r},  \Omega\bigg)= O\bigl( e^{-(\frac {p+1}2 -\sigma)|x^{+}_{2}-x^{+}_{1}|}\bigr),
			\end{split}
		\end{equation}
		\begin{equation}\label{h1}
			\begin{split}
				I_\ell\bigg(U_{x^{+}_{j}} ,  \frac{\partial U_{x^{+}_{j}}}{\partial h},  \Omega\bigg)= O\bigl(r_k e^{-(\frac {p+1}2 -\sigma)|x^{+}_{2}-x^{+}_{1}|}\bigr).
			\end{split}
		\end{equation}
		For $i\ne j$, we have
		
		\begin{equation}\label{5-20-3}
			\begin{split}
				I_\ell\bigg(U_{x^{+}_{i}} ,  \frac{\partial U_{x^{+}_{j}}}{\partial r},  \Omega\bigg)
				=p \int_\Omega U_{x^{+}_{j}}^{p-1}\frac{\partial U_{x^{+}_{j}}}{\partial r}\frac{\partial U_{x^{+}_{i}} }{\partial y_\ell}  -p \int_\Omega U_{x^{+}_{i}}^{p-1}\frac{\partial U_{x^{+}_{j}}}{\partial r}\frac{\partial U_{x^{+}_{i}} }{\partial y_\ell}+O\bigl( e^{-(\frac {p+1}2 -\sigma) |x^{+}_{2}-x^{+}_{1}|}\bigr),
			\end{split}
		\end{equation}
		\begin{equation}\label{h2}
			\begin{split}
				I_\ell\bigg(U_{x^{+}_{i}} ,  \frac{\partial U_{x^{+}_{j}}}{\partial h},  \Omega\bigg)
				=p \int_\Omega U_{x^{+}_{j}}^{p-1}\frac{\partial U_{x^{+}_{j}}}{\partial h}\frac{\partial U_{x^{+}_{i}} }{\partial y_\ell}  -p \int_\Omega U_{x^{+}_{i}}^{p-1}\frac{\partial U_{x^{+}_{j}}}{\partial h}\frac{\partial U_{x^{+}_{i}} }{\partial y_\ell}+O\bigl(r_k e^{-(\frac {p+1}2 -\sigma) |x^{+}_{2}-x^{+}_{1}|}\bigr).
			\end{split}
		\end{equation}
		And for $j=1,\cdots,k$, it holds that
		\begin{equation}\label{5-20-4}
			\begin{split}
				I_\ell\bigg(U_{x^{+}_{j}} ,  \frac{\partial U_{x^{-}_{j}}}{\partial r},  \Omega\bigg)
				=p \int_\Omega U_{x^{-}_{j}}^{p-1}\frac{\partial U_{x^{-}_{j}}}{\partial r}\frac{\partial U_{x^{+}_{j}} }{\partial y_\ell}  -p \int_\Omega U_{x^{+}_{j}}^{p-1}\frac{\partial U_{x^{-}_{j}}}{\partial r}\frac{\partial U_{x^{+}_{j}} }{\partial y_\ell}+O\bigl( e^{-(\frac {p+1}2 -\sigma) |x^{+}_{2}-x^{+}_{1}|}\bigr),
			\end{split}
		\end{equation}
		\begin{equation}\label{h3}
			\begin{split}
				I_\ell\bigg(U_{x^{+}_{j}} ,  \frac{\partial U_{x^{-}_{j}}}{\partial h},  \Omega\bigg)
				=p \int_\Omega U_{x^{-}_{j}}^{p-1}\frac{\partial U_{x^{-}_{j}}}{\partial h}\frac{\partial U_{x^{+}_{j}} }{\partial y_\ell}  -p \int_\Omega U_{x^{+}_{j}}^{p-1}\frac{\partial U_{x^{-}_{j}}}{\partial h}\frac{\partial U_{x^{+}_{j}} }{\partial y_\ell}+O\bigl( r_k e^{-(\frac {p+1}2 -\sigma) |x^{+}_{2}-x^{+}_{1}|}\bigr),
			\end{split}
		\end{equation}
		where  $\ell=1,3$ and $\sigma>0$ is any fixed small constant.
	\end{lemma}
	
	\begin{proof}
		For brevity, we only prove  the the case when $\ell=1$. According to \eqref{1-27-11} and \eqref{10-13-12}, we have for $i=j$,
		\begin{equation}
			\begin{split}
				I_1\bigg(U_{x^{+}_{i}} ,  \frac{\partial U_{x^{+}_{j}}}{\partial r},  \Omega\bigg)
				= &\int_\Omega \bigl( -\Delta U_{x^{+}_{i}} +   U_{x^{+}_{i}} - U_{x^{+}_{i}}^p\bigr) \frac{\partial \frac{\partial U_{x_{j}^{+}}}{\partial r}}{\partial y_1}\\
				& +
				\int_\Omega \bigl( -\Delta \frac{\partial U_{x^{+}_{j}}}{\partial r}  +  \frac{\partial U_{x^{+}_{j}}}{\partial r} -p U_{x^{+}_{i}}^{p-1}\frac{\partial U_{x^{+}_{j}}}{\partial r}\bigr)\frac{\partial U_{x^{+}_{i}} }{\partial y_1}+\int_{\partial \Omega} U_{x^{+}_{i}}^p \frac{\partial U_{x^{+}_{j}}}{\partial r} \nu_1\\
				=&\int_{\partial \Omega} U_{x^{+}_{j}}^p \frac{\partial U_{x^{+}_{j}}}{\partial r} \nu_1= O\bigl( e^{-(\frac {p+1}2 -\sigma)|x^{+}_{2}-x^{+}_{1}|}\bigr),
			\end{split}
		\end{equation}
		While if $i\ne j$, we have
		\begin{equation}
			\begin{split}
				I_1\bigg(U_{x^{+}_{i}}& ,  \frac{\partial U_{x^{+}_{j}}}{\partial r},  \Omega\bigg)
				\\
				= &
				\int_\Omega \bigl( -\Delta \frac{\partial U_{x^{+}_{j}}}{\partial r}  +  \frac{\partial U_{x^{+}_{j}}}{\partial r} -p U_{x^{+}_{i}}^{p-1}\frac{\partial U_{x^{+}_{j}}}{\partial r}\bigr)\frac{\partial U_{x^{+}_{i}} }{\partial y_1}
				+   O\bigl( e^{-(\frac {p+1}2 -\sigma)|x^{+}_{2}-x^{+}_{1}|}\bigr)\\
				=&p \int_\Omega U_{x^{+}_{j}}^{p-1}\frac{\partial U_{x^{+}_{j}}}{\partial r}\frac{\partial U_{x^{+}_{i}} }{\partial y_1}  -p \int_\Omega U_{x^{+}_{i}}^{p-1}\frac{\partial U_{x^{+}_{j}}}{\partial r}\frac{\partial U_{x^{+}_{i}} }{\partial y_1}+O\bigl( e^{-(\frac {p+1}2 -\sigma) |x^{+}_{2}-x^{+}_{1}|}\bigr),
			\end{split}
		\end{equation}
		and for $j=1,\cdots,k$, we have
		\begin{equation}
			\begin{split}
				& I_1\bigg(U_{x^{+}_{j}} ,  \frac{\partial U_{x^{-}_{j}}}{\partial r},  \Omega\bigg)
				\\
				=&p \int_\Omega U_{x^{-}_{j}}^{p-1}\frac{\partial U_{x^{-}_{j}}}{\partial r}\frac{\partial U_{x^{+}_{j}} }{\partial y_1}  -p \int_\Omega U_{x^{+}_{j}}^{p-1}\frac{\partial U_{x^{-}_{j}}}{\partial r}\frac{\partial U_{x^{+}_{j}} }{\partial y_1}+O\bigl( e^{-(\frac {p+1}2 -\sigma) |x^{+}_{2}-x^{+}_{1}|}+e^{-(\frac {p+1}2 -\sigma) |x^{+}_{1}-x^{-}_{1}|}\bigr).
			\end{split}
		\end{equation}
		
		Similarly, we calculate it for $i=j$,
		\begin{equation}
			\begin{split}
				& I_1\bigg(U_{x^{+}_{i}} ,  \frac{\partial U_{x^{+}_{j}}}{\partial h},  \Omega\bigg)= \int_{\partial \Omega} U_{x^{+}_{j}}^p \frac{\partial U_{x^{+}_{j}}}{\partial h} \nu_1= O\bigl(r_k e^{-(\frac {p+1}2 -\sigma)|x^{+}_{2}-x^{+}_{1}|}\bigr),
			\end{split}
		\end{equation}
		For $i\ne j$, we have
		\begin{equation}
			\begin{split}
				I_1\bigg(U_{x^{+}_{i}} , & \frac{\partial U_{x^{+}_{j}}}{\partial h},  \Omega\bigg)
				\\
				= &
				\int_\Omega \bigl( -\Delta \frac{\partial U_{x^{+}_{j}}}{\partial h}  +  \frac{\partial U_{x^{+}_{j}}}{\partial h} -p U_{x^{+}_{i}}^{p-1}\frac{\partial U_{x^{+}_{j}}}{\partial h}\bigr)\frac{\partial U_{x^{+}_{i}} }{\partial y_1}
				+   O\bigl(r_k e^{-(\frac {p+1}2 -\sigma)|x^{+}_{2}-x^{+}_{1}|}\bigr)\\
				=&p \int_\Omega U_{x^{+}_{j}}^{p-1}\frac{\partial U_{x^{+}_{j}}}{\partial h}\frac{\partial U_{x^{+}_{i}} }{\partial y_1}  -p \int_\Omega U_{x^{+}_{i}}^{p-1}\frac{\partial U_{x^{+}_{j}}}{\partial h}\frac{\partial U_{x^{+}_{i}} }{\partial y_1}+O\bigl(r_k e^{-(\frac {p+1}2 -\sigma) |x^{+}_{2}-x^{+}_{1}|}\bigr),
			\end{split}
		\end{equation}
		and for $j=1,\cdots,k$, we have
		\begin{equation}
			\begin{split}
				& I_1\bigg(U_{x^{+}_{j}} ,  \frac{\partial U_{x^{-}_{j}}}{\partial h},  \Omega\bigg)
				\\
				=&p \int_\Omega U_{x^{-}_{j}}^{p-1}\frac{\partial U_{x^{-}_{j}}}{\partial h}\frac{\partial U_{x^{+}_{j}} }{\partial y_1}  -p \int_\Omega U_{x^{+}_{j}}^{p-1}\frac{\partial U_{x^{-}_{j}}}{\partial h}\frac{\partial U_{x^{+}_{j}} }{\partial y_1}+O\bigl( r_k e^{-(\frac {p+1}2 -\sigma) |x^{+}_{2}-x^{+}_{1}|}+r_k e^{-(\frac {p+1}2 -\sigma) |x^{+}_{1}-x^{-}_{1}|}\bigr).
			\end{split}
		\end{equation}

	\end{proof}

\end{document}